\documentclass[a4paper,10pt]{amsart}
\newcommand{\thelanguage}{english}
\usepackage[\thelanguage]{babel}

\usepackage{geometry}
\usepackage{mathrsfs}
\usepackage{tkz-euclide, tikz, pgfplots}
\usetikzlibrary{backgrounds,patterns,matrix,calc,arrows,positioning,fit,shapes.geometric,decorations.pathmorphing}
\usepackage{tikz-cd}
\usepackage{float}
\usepackage{amsthm,amsmath,amssymb,amsfonts}
\usepackage[colorlinks=true, citecolor=blue]{hyperref}
\usepackage{multirow}
\usepackage{enumitem}
\usepackage{tabularx}
\usepackage{booktabs}
\usepackage[sort]{cite}
\usepackage[\thelanguage]{cleveref}

\usepackage{marginnote}
\usepackage{etex}
\usepackage{subfig}

\DeclareMathAlphabet{\mathpzc}{OT1}{pzc}{m}{it}
\newcommand{\pzc}[1]{\mathpzc{#1}}

\newcommand{\forcebold}[1]{\boldsymbol{#1}}
\newcommand{\tuple}[1]{\mathfrak{#1}}

\newcommand{\Var}[1]{\mathcal{#1}}
\newcommand{\Sec}[2]{\sigma_{#1}(#2)}

\newcommand{\tensor}[1]{\pzc{#1}}
\newcommand{\vect}[1]{\mathbf{#1}}
\newcommand{\sten}[3]{\vect{#1}_{#2}^{#3}}
\newcommand{\Tang}[2]{\mathrm{T}_{#1} {#2}}

\newcommand{\diag}{\operatorname{diag}}

\newcommand{\R}{\mathbb{R}}
\newcommand{\deriv}[2]{\mathrm{d}_{#2} #1}

\numberwithin{equation}{section}
\numberwithin{figure}{section}
\numberwithin{table}{section}

\theoremstyle{plain}
\newcounter{numbering} \numberwithin{numbering}{section}
\newtheorem{thm}[numbering]{Theorem}
\newtheorem{lemma}[numbering]{Lemma}
\newtheorem{prop}[numbering]{Proposition}
\newtheorem{cor}[numbering]{Corollary}
\theoremstyle{definition}

\newtheorem{dfn}[numbering]{Definition}

\theoremstyle{remark}

\newtheorem{rem}[numbering]{Remark}

\crefname{equation}{}{}
\crefname{equation}{}{}
\crefname{figure}{Figure}{Figures}
\crefname{section}{Section}{Sections}
\crefname{table}{Table}{Tables}
\crefname{lemma}{Lemma}{Lemmata}
\crefname{prop}{Proposition}{Propositions}
\crefname{thm}{Theorem}{Theorems}
\crefname{cor}{Corollary}{Corollaries}
\crefname{dfn}{Definition}{Definitions}
\crefname{hyp}{Hypothesis}{Hypotheses}
\crefname{notation}{Notations}{Notations}
\crefname{rem}{Remark}{Remarks}
\crefname{claim}{Claim}{claims}
\newcommand{\C}{\mathbb{C}}

\newcommand{\refthm}[1]{{\cref{#1}}}
\newcommand{\reflem}[1]{{\cref{#1}}}
\newcommand{\refeqn}[1]{{\cref{#1}}}

\allowdisplaybreaks

\title[Pencil-based algorithms for CPD are unstable]{Pencil-based algorithms for tensor rank decomposition are not stable}
\author{Carlos Beltr\'an}
\thanks{CB: Universidad de Cantabria, beltranc@unican.es. Supported by Spanish
	``Ministerio de Econom\'ia y Competitividad'' under projects MTM2017-83816-P
	and MTM2017-90682-REDT (Red ALAMA), as well as by the Banco Santander and Universidad de Cantabria under
	project 21.SI01.64658.}
\author{Paul Breiding}
\thanks{PB: Max-Planck-Institute for Mathematics in the Sciences Leipzig, breiding@mis.mpg.de.}
\author{Nick Vannieuwenhoven}
\thanks{NV: KU Leuven, Department of Computer Science, nick.vannieuwenhoven@kuleuven.be. Supported by a Postdoctoral Fellowship of the Research Foundation--Flanders (FWO)}

\begin{document}
\begin{abstract}
We prove the existence of an open set of $n_1\times n_2 \times n_3$ tensors of rank $r$ on which a popular and efficient class of algorithms for computing tensor rank decompositions based on a reduction to a linear matrix pencil, typically followed by a generalized eigendecomposition, is arbitrarily numerically forward unstable. Our analysis shows that this problem is caused by the fact that the condition number of the tensor rank decomposition can be much larger for $n_1 \times n_2 \times 2$ tensors than for the $n_1\times n_2 \times n_3$ input tensor. Moreover, we present a lower bound for the limiting distribution of the condition number of random tensor rank decompositions of third-order tensors. The numerical experiments illustrate that for random tensor rank decompositions one should anticipate a loss of precision of a few digits.
\end{abstract}



\maketitle

\section{Introduction} \label{sec_introduction}
We study the numerical stability of one of the most popular and effective class of algorithms for computing the \textit{tensor rank decomposition}, or \textit{canonical polyadic decomposition} (CPD), of a tensor.
Recall that a \textit{rank-$1$ tensor} is represented by a multidimensional $n_1 \times n_2 \times \cdots \times n_d$ array~$\tensor{B} = (b_{i_1,i_2,\ldots,i_d})_{1\leq i_1\leq n_1, \ldots, 1\leq i_d\leq n_d}$ whose elements satisfy the following property:
\[
 b_{i_1,i_2,\ldots,i_d} = b_{i_1}^{(1)} b_{i_2}^{(2)} \cdots b_{i_d}^{(d)}, \text{ where } \vect{b}^{k} = (b_i^{(k)})_{i=1}^{n_k} \in \R^{n_k}.
\]
For brevity, one writes $\tensor{B} = \vect{b}^{1} \otimes \vect{b}^{2} \otimes \cdots \otimes \vect{b}^{d}$.
The CPD of $\tensor{A} \in \R^{n_1 \times \cdots \times n_d}$ was proposed by Hitchcock \cite{hitchcock}. It expresses $\tensor{A}$ as a minimum-length linear combination of rank-$1$ tensors:
\begin{equation} \label{CPD}
 \tensor{A} = \tensor{A}_1+\tensor{A}_2 + \cdots+\tensor{A}_r,\quad\text{where } \tensor{A}_i = \sten{a}{i}{1} \otimes \sten{a}{i}{2} \otimes \cdots \otimes \sten{a}{i}{d} \text{ and } \sten{a}{i}{k} \in \R^{n_k}
\end{equation}
for all $i=1,\ldots,r$ and $k=1,\ldots,d$.
The number $r$ in \refeqn{CPD} is called the \emph{rank} and $d$ is the \textit{order} of $\tensor{A}$. It is often convenient to consider the \textit{factor matrices} $A_1,\ldots,A_d$, where $A_k := [\sten{a}{i}{k}]_{i=1}^r$.

Mainly due to its simplicity and uniqueness properties \cite{Kruskal1977,COV2014},
the CPD has found application in a diverse set of scientific fields; see \cite{Kroonenberg2008,SBG2004,BCS1997,Comon1994,CJ2010,Review2016,KB2009}.
A rank-$r$ tensor $\tensor{A}$ is called \textit{$r$-identifiable} if the set of rank-$1$ tensors $\{\tensor{A}_1, \tensor{A}_2, \ldots, \tensor{A}_r\}$ whose sum is $\tensor{A}$, as in \cref{CPD}, is uniquely determined given $\tensor{A}$.
A classic result result on $r$-identifiability is Kruskal's criterion \cite{Kruskal1977}. It is formulated in terms of the \textit{Kruskal rank} $k_M$ of a matrix $M$: $k_M$ is the largest integer $k$ such that every subset of $k$ columns of $M$ has rank equal to $k$.
\begin{lemma}[Kruskal's criterion] \label{kruskal}
Let $\tensor{A} = \sum_{i=1}^r \sten{a}{i}{} \otimes \sten{b}{i}{}\otimes\sten{c}{i}{}$ be a tensor with factor matrices $A=[\vect{a}_i]_i$, $B=[\vect{b}_i]_i$ and $C=[\vect{c}_i]_i$. A sufficient condition for the $r$-identifiability of $\tensor{A}$ is
\(
 r \le \frac{1}{2}( k_A + k_B + k_C - 2)
\)
and $k_A, k_B, k_C > 1$.
\end{lemma}

Most low-rank tensors satisfy Kruskal's criterion; more precisely, there is an open dense subset of the set of rank-$r$ tensors in $\R^{n_1 \times n_2 \times n_3}$, $n_1 \ge n_2 \ge n_3 \ge 2$, where $r$-identifiability holds, provided that $r \le n_1 + \min\{ \tfrac{1}{2} \delta, \delta\}$ with $\delta := n_2 + n_3 - n_1 - 2$.

The computational problem of recovering the set of rank-$1$ tensors $\{ \tensor{A}_1, \ldots, \tensor{A}_r \}$ whose sum is $\tensor{A}$ is called the \emph{tensor rank decomposition problem} (TDP). When the rank of a third-order tensor is sufficiently small, there are efficient, numerical, direct algorithms for solving the TDP, such as those in \cite{SY1980,Lorber1985,SK1990,LRA1993,FFB2001,DdL2014,DdL2017}. All of these algorithms involve the computation of a \textit{generalized eigendecomposition} (GEVD) of a linear \textit{matrix pencil} constructed from the low-rank input tensor. An algorithm for solving TDPs that involves such a reduction to a matrix pencil will subsequently be called a \textit{pencil-based algorithm} (PBA). This will be given a precise meaning in \cref{def_pencil_algorithm}, where we rigorously define the class of PBAs.

A prototypical example of a PBA is presented next. The essential idea is to project a given tensor $\tensor{A} \in \R^{n_1 \times n_2 \times n_3}$, $n_1 \ge n_2 \ge r$, to a tensor of format $n_1\times n_2\times 2$ and recover the first factor matrix from the latter. The input $\tensor{A} = \sum_{i=1}^r {\vect{a}}_i \otimes {\vect{b}}_i \otimes \vect{c}_i$ is assumed to admit a unique rank-$r$ CPD with $\|\vect{a}_i\|=1$ for all $i=1,\ldots,r$. Let $Q\in\R^{n_3 \times 2}$ be a matrix with orthonormal columns. Then, contracting $\tensor{A}$ along the third mode by $Q^T$, which is a special type of \emph{multilinear multiplication} \cite{KB2009,dSL2008}, yields the tensor
\[
\tensor{B} = (I_{n_1}, I_{n_2}, Q^T) \cdot \tensor{A} := \sum_{i=1}^r \vect{a}_i \otimes \vect{b}_i \otimes \vect{z}_i\in\R^{n_1\times n_2\times 2}, \text{ where } \vect{z}_i = Q^T \vect{c}_i
\]
and $I_m$ denotes the $m \times m$ identity matrix. Let $Q_1 \in \R^{n_1 \times r}$, respectively $Q_2 \in \R^{n_2 \times r}$, be a matrix with orthonormal columns that form a basis for $\{\vect{a}_i\}$, respectively $\{\vect{b}_i\}$. The following is then a specific orthogonal Tucker decomposition \cite{Tucker1966} of $\tensor{B}$:
\[
\tensor{B} := (Q_1, Q_2, I) \cdot \tensor{S} := \sum_{i=1}^r (Q_1 \vect{x}_1') \otimes (Q_2 \vect{y}_i') \otimes \vect{z}_i, \text{ where } \vect{x}_i' = Q_1^T \vect{a}_i \text{ and } \vect{y}_i' = Q_2^T \vect{b}_i.
\]
Let $X = [ \tfrac{\vect{x}_i'}{\|\vect{x}_i'\|}]_{1\leq i\leq r}$ and $Y = [\tfrac{\vect{y}_i'}{\|\vect{y}_i'\|}]_{1\leq i\leq r}$. Then it follows from the properties of multilinear multiplication that the \textit{core tensor} $\tensor{S} = (Q_1^T, Q_2^T, I_{2}) \cdot \tensor{B} \in \R^{r \times r \times 2}$ has the following two \textit{$3$-slices}:
\[
 S_j := (I,I,\vect{e}_j)^T \cdot \tensor{S}
 := \sum_{i=1}^r \lambda_{j,i} \cdot \vect{x}_i \otimes \vect{y}_i
 = \sum_{i=1}^r \lambda_{j,i} \cdot \vect{x}_i \vect{y}_i^T
 = X \operatorname{diag}( \forcebold{\lambda}_j ) Y^T, \quad j=1,2,
\]
where $\forcebold{\lambda}_{j} := [z_{j,i} \|\vect{x}_i'\| \|\vect{y}_i'\|]_{i=1}^r$.
Whenever $S_1$ and $S_2$ are nonsingular, we have
\[
 S_1 S_2^{-1} = X \operatorname{diag}(\forcebold{\lambda}_1) \operatorname{diag}(\forcebold{\lambda}_2)^{-1} X^{-1};
\]
thus $X$ is the matrix of eigenvectors of the GEVD of the nonsingular matrix pencil $(S_1, S_2)$. As long as the eigenvalues are distinct, the matrix $X$ is uniquely determined and it follows that $A = Q_1 X$. Finally, the rank-$1$ tensors $\tensor{A}_i = \vect{a}_i \otimes \vect{b}_i \otimes \vect{c}_i$ are recovered by the following well-known property \cite{KB2009,Review2016} of the $1$-flattening: $\tensor{A}_{(1)} = A (B \odot C)^T$, where $M \odot N := [\vect{m}_i \otimes \vect{n}_i]_{i=1}^r \in \R^{mn \times r}$ is the Khatri--Rao product of $M \in \R^{m \times r}$ and $N \in \R^{n \times r}$. Then, we see that
\[
 A \odot (A^\dagger \tensor{A}_{(1)})^T = A \odot (B \odot C) = A \odot B \odot C = \begin{bmatrix} \tensor{A}_1 & \tensor{A}_2 & \cdots & \tensor{A}_r \end{bmatrix},
\]
where $X^\dagger$ is the Moore--Penrose pseudoinverse of $X$. This procedure thus solves the TDP.

The above algorithm and those in \cite{SY1980,Lorber1985,SK1990,LRA1993,FFB2001,DdL2014,DdL2017} have the major advantage that the CPD can be computed via a sequence of numerically stable and efficient linear algebra algorithms for solving classic problems such as linear system solving, linear least-squares and generalized eigendecomposition problems. In light of the plentiful indications that computing a CPD is a difficult problem---the NP-completeness of tensor rank \cite{Hastad1990}, the ill-posedness of the corresponding approximation problem \cite{dSL2008}, and the potential (average) ill-conditioning of the TDP \cite{BV2018c,BV2017}---the existence of aforementioned algorithms is almost too good to be true. We show that there is a price to be paid in the currency of the achievable precision by establishing the following result.

\begin{thm} \label{thm_unstable}
Let $n_1 \ge n_2 \ge n_3 > r+1 \ge 2$. For every pencil-based algorithm, there exists an open set of the rank-$r$ tensors in $\R^{n_1 \times n_2 \times n_3}$ for which it is unstable.
\end{thm}

The instability in the theorem is with respect to the standard model of floating-point arithmetic \cite{higham}, namely
\[
\operatorname{fl}( a ) = (1+\delta)(a) \;\text{ and }\;
 \operatorname{fl}( a \circ b ) = (1 + \delta)(a \circ b), \;|\delta| \le \epsilon_u,\; \text{where } \circ \in \{ +, -, \cdot, / \},
\]
where $\operatorname{fl}(a)$ denotes the floating-point representation of $a$, and $\epsilon_u$ is the unit roundoff. In IEEE double-precision floating-point arithmetic $\epsilon_u \approx 1.11 \cdot 10^{-16}$ \cite[Chapter 2]{higham}.

In practice, \refthm{thm_unstable} covers the algorithms from \cite{SY1980,Lorber1985,SK1990,LRA1993,FFB2001}, \texttt{cpd\_gevd} from Tensorlab~v3.0 \cite{Tensorlab}, \cite[Algorithm 2]{DdL2014}, {and the foregoing prototypical PBA.}
Algorithm 1 of \cite{DdL2014}, as well as both algorithms in \cite{DdL2017}, are likely also unstable because they use an unstable algorithm in intermediate steps; a more thorough analysis would be required to show this rigorously.

\begin{rem}
For higher-order tensors $\tensor{A} \in \R^{n_1 \times \cdots \times n_d}$ with $d \ge 4$ it is a common practice to \emph{reshape} them into a third-order tensor $\tensor{A}_{(\vect{j},\vect{k},\vect{l})} \in \R^{m_1 \times m_2 \times m_3}$ by choosing a partition of the indices $\{1,\ldots,d\}=\{j_1,\ldots,j_s\} \sqcup \{k_1,\ldots,k_t\} \sqcup \{l_1,\ldots,l_u\}$ with $m_1 = j_1\cdots j_s$, $m_2 = k_1\cdots k_t$, and $m_3 = l_1\cdots l_u$. Under the conditions of section 7 of \cite{COV2017}, the CPD of $\tensor{A}_{(\vect{j},\vect{k},\vect{l})}$, i.e., the set of rank-$1$ tensors, can be reshaped back into a set of order-$d$ tensors in $\R^{n_1 \times \cdots \times n_d}$ yielding the CPD of $\tensor{A}$. According to Theorem \ref{thm_unstable} this strategy employs an unstable algorithm as intermediate step, so we should \emph{a priori} expect that the resulting algorithm is also unstable. This can be proved rigorously for $u = |\vect{l}| = 1$ by a slight generalization of the argument in \cref{sec_3o_unstable}. We leave a general proof as an open question.
\end{rem}

It is important to mention that the stabilities of algorithms employed in the intermediate steps of a PBA are \textit{not} the reason why PBAs are unstable. In the above prototypical PBA, all individual steps can be implemented using numerically stable algorithms, but the resulting algorithm is nevertheless unstable. The instability in \cref{thm_unstable} is caused by a large difference
between the \textit{condition numbers} of the TDPs in $\R^{n_1 \times n_2 \times n_3}$ and $\R^{n_1 \times n_2 \times 2}$.

The condition number of the TDP was studied in \cite{BV2017}.\footnote{A condition number of the different problem of computing the \emph{factor matrices} was considered in \cite{V2017}.} Let us denote the set of $n_1\times \cdots\times n_d$ tensors of rank 1 by $\Var S$. This set is actually a smooth manifold, called the \emph{Segre manifold}; see \cref{sec_diff_geom}. Tensors of rank at most $r$ are obtained as the image of the \textit{addition map}
$\Phi_r:\Var{S}^{\times r}\to \R^{n_1\times \cdots \times n_d}, (\tensor A_1,\ldots, \tensor A_r)\to \tensor A_1 + \cdots + \tensor A_r$. The condition number of the TDP at a rank-$r$ tensor $\tensor{A}$ with {\emph{ordered CPD}} $\tuple{a} = (\tensor{A}_1,\ldots, \tensor{A}_r)$ is
\begin{equation}\label{def_kappa}
\kappa(\tensor{A},(\tensor A_1,\ldots, \tensor A_r)) = \lim\limits_{\epsilon \to 0} \,\sup\limits_{\substack{\tensor{B} \text{ has rank } r, \\ \Vert \tensor{A} - \tensor{B}\Vert_F < \epsilon}} \,\frac{\Vert\Phi_\tuple{a}^{-1}(\tensor{A}) - \Phi_\tuple{a}^{-1}(\tensor{B})\Vert_F}{\Vert \tensor{A} - \tensor{B}\Vert_F},
\end{equation}
where $\Phi_\tuple{a}^{-1}$ is the local inverse function of $\Phi_r$ that satisfies $\Phi_\tuple{a}^{-1}(\tensor{A}) = {(\tensor A_1,\ldots, \tensor A_r)}$; see \cite{BV2017}. The norms are the Euclidean norms on the ambient spaces of domain and image of $\Phi_r$, which is naturally identified with the Frobenius norms of tensors, i.e., the square root of the sum of squares of the elements. It follows from the spectral characterization in \cite[Theorem 1.1]{BV2017} that $\tensor{A}$ depends uniquely on \emph{the (unordered) CPD} $\{ \tensor A_1,\ldots, \tensor A_r \}$; therefore we often write $\kappa(\tensor A_1,\ldots, \tensor A_r)$ for the condition number. If such a local inverse does not exist, we have $\kappa(\tensor A_1,\ldots, \tensor A_r) := \infty$. In \cref{sec_diff_geom} we discuss in more detail the existence of this local inverse function; it will be shown in \cref{prop_image} that ``most tensors have a finite condition number.''

While the proof of \cref{thm_unstable} is not straightforward,
the main intuition that led us to its conception is the observation that there appears to be a gap in the expected value of the condition number of TDPs in $\R^{m_1 \times m_2 \times 2}$ and other spaces $\R^{m_1 \times m_2 \times m_3}$, $m_1 \ge m_2 \ge m_3 \ge 3$, as we observed in \cite{BV2018c}. Here, we derived a further characterization of the \emph{distribution} of the condition number of random CPDs, based on a result of Cai, Fan, and Jiang \cite{CFJ13} about the distribution of the minimum distance between random points on spheres.

\begin{thm} \label{thm_probability_dist_est}
Let $\sten{a}{1}{}, \ldots ,\sten{a}{r}{} \in \R^{m_1}$, $\sten{b}{1}{}, \ldots ,\sten{b}{r}{} \in \R^{m_2}$ be arbitrary and fixed, while we assume that $\sten{c}{1}{}, \ldots ,\sten{c}{r}{} \in \R^{m_3}$ are independent random vectors with standard normal entries. Consider the random rank-$1$ tensors $\tensor{A}_i = \sten{a}{i}{} \otimes \sten{b}{i}{}\otimes \sten{c}{i}{}\in \R^{m_1 \times m_2 \times m_3}$. Then, for any $\alpha > 0$ we have
\[
  \mathrm{P}\left[\kappa(\tensor{A}_1,\ldots,\tensor{A}_r)\geq \alpha r^{\frac{2}{m_3-1}}\right]\geq T_{r,\alpha}, \quad\text{ where } \lim\limits_{r\to \infty} T_{r,\alpha} = 1 - e^{ -K \alpha^{1-m_3}};
  \]
herein,
\(
 K = \frac{2^{\frac{1}{2} (m_3-5)}}{\sqrt{\pi}} \frac{\Gamma(\frac{m_3}{2})}{\Gamma(\frac{m_3+1}{2})},
\)
where $\Gamma$ is the gamma function. In particular, if $m_3=2$ we have
\[
\mathrm{P}\left[\kappa(\tensor{A}_1,\ldots,\tensor{A}_r)\geq \alpha r^2\right]\geq T_{r,\alpha}, \quad\text{ where } \lim\limits_{r\to \infty} T_{r,\alpha} = 1 - e^{ -\frac{1}{\sqrt{2}\,\pi\alpha}}\approx \frac{1}{\sqrt{2}\,\pi\alpha}.
\]
\end{thm}

This theorem suggests that \textit{as $m_3$ increases, very large condition numbers become increasingly unlikely}. The worst case thus seems to occur for $m_3 = 2$, which is exactly the space from which PBAs try to recover the CPD.
For example, if $m_3=2$ and $r$ is large we can expect that the condition number is greater than $4r^2$ with probability at least (around) $5\%$.

\subsection*{Outline} The next section recalls some preliminary material.
As \cref{thm_probability_dist_est} provides the main intuition for the main result, we will treat it first in \cref{sec_distribution}. Before proving \cref{thm_unstable}, we need a precise definition of a PBA. This definition relies on the notion of \emph{$r$-nice tensors} that we study in \cref{sec_nice}; these rank-$r$ tensors have convenient differential-geometric properties.
Then, in \cref{sec_pencilbased} we define the class of PBAs.
\Cref{sec_unstable_algorithm} is dedicated to the proof of \cref{thm_unstable}.
Numerical experiments validating the theory and illustrating typical behavior for random CPDs are presented in \cref{sec_numex}. Finally, \cref{sec_conclusions} presents our main conclusions.

\subsection*{Notation}
The following notational conventions are observed throughout this paper: scalars are typeset in lower-case letters ($a$), vectors in bold-face lower-case letters ($\vect{a}$), matrices in upper-case letters ($A$), tensors in a calligraphic font ($\tensor{A}$), and varieties and manifolds in an alternative calligraphic font ($\mathcal{A}$). The unit sphere over a set $V \subset \R^m$ is $\mathbb{S}(V) := \{ \vect{v} \mid \vect{v} \in V, \|\vect{v}\| = 1\}$. The Moore--Penrose pseudoinverse of a matrix $M \in \R^{m \times n}$ is denoted by $M^\dagger$.
The $m \times m$ identity matrix is denoted by $I_m$. The symmetric group of permutations on $r$ elements is denoted by $\mathfrak{S}_r$. $P_\pi$ denotes the $r \times r$ permutation matrix representing the permutation $\pi \in \mathfrak{S}_r$. The standard Euclidean inner product on $\R^m$ is $\langle \vect{x}, \vect{y} \rangle := \vect{x}^T \vect{y}$ for $\vect{x}, \vect{y} \in \R^m$.

\subsection*{Acknowledgements}
We thank Vanni Noferini and Leonardo Robol for interesting discussions on the definition of numerical instability.

\section{Preliminaries} \label{sec_preliminaries}
Some elementary definitions from multilinear algebra and differential geometry are recalled.

\subsection{Multilinear algebra}
The tensor product $\otimes$ of vector spaces $V_1, \ldots, V_d$ is denoted by $\otimes$; see \cite[Chapter 1]{Greub1978}. As the tensor product is unique up to isomorphisms of the vector spaces~$V_1 \times \cdots \times V_d$ and $V_1 \otimes \cdots \otimes V_d$, we will be particularly liberal between the interpretations $\R^{n_1} \otimes \cdots \otimes \R^{n_d} \simeq \R^{n_1 \times \cdots \times n_d} \simeq \R^{n_1 \cdots n_d}$. Elements in the first space are abstract order-$d$ tensors, in the second space they are $d$-arrays, while in the last space they are long vectors. We do not use a ``vectorization'' operator to indicate the natural bijection between the last two spaces.

The tensor product of linear maps is also well defined \cite[Chapter 1]{Greub1978}. We use this definition in expressions $M_1 \otimes \cdots \otimes M_d$, where $M_k = [\vect{m}^k_i]_i \in \R^{m_k \times n_k}$, whose columns are $\vect{m}_{i_1}^1 \otimes \cdots \otimes \vect{m}_{i_d}^d$; the order will not be relevant wherever it is used. The \textit{multilinear multiplication} of a tensor $\tensor{A}= \sum_{i_1,\ldots,i_d} a_{i_1,\ldots,i_d} \vect{e}_{i_1}^1 \otimes \cdots \otimes \vect{e}_{i_d}^d\in \R^{n_1 \times \cdots \times n_d}$ with the above matrices $M_k$ is
\[
(M_1, \ldots, M_d) \cdot \tensor{A}
:= (M_1 \otimes \cdots \otimes M_d)(\tensor{A})
= \sum_{i_1=1}^{n_1} \cdots \sum_{i_d=1}^{n_d} a_{i_1, \ldots, a_d} (M_1 \vect{e}_{i_1}^1) \otimes \cdots \otimes (M_d \vect{e}_{i_d}^d).
\]
This also entails the following well-known formula for the inner product between rank-$1$ tensors:
\begin{equation}\label{eq:innerprod}
 \langle \vect{a}_1 \otimes \cdots \otimes \vect{a}_d, \vect{b}_1 \otimes \cdots \otimes \vect{b}_d \rangle
 = \prod_{k=1}^d \langle \vect{a}_k, \vect{b}_k \rangle;
\end{equation}
see, e.g., \cite[Section 4.5]{Hackbusch2012}. The Khatri--Rao product of the matrices $M_k = [\vect{m}_i^k]_i \in \R^{n_k \times r}$ is
\begin{equation*} 
	M_1 \odot \cdots \odot M_d := [ \vect{m}_i^1 \otimes \cdots \otimes \vect{m}_i^d ]_i \in \R^{n_1 \cdots n_d \times r}.
\end{equation*}
Note that it is a subset of columns from the tensor product $M_1 \otimes \cdots \otimes M_d$.

\subsection{Differential geometry}
The following elementary definitions are presented here only for submanifolds of Euclidean spaces; see, e.g., \cite{Lee2013} for the general definitions.
By a smooth ($C^\infty$) manifold we mean a topological manifold with a smooth structure, in the sense of \cite{Lee2013}.
The \textit{tangent space} at $x$ to an $n$-dimensional smooth submanifold $\Var{M} \subset \R^N$ can be defined as
\[
\Tang{x}{\Var{M}} = \left\{ \vect{v} \in \R^N \;|\; \exists \text{ a smooth curve } \gamma(t) \subset \Var{M} \text{ with }
\gamma(0)=x: \vect{v} = \frac{\mathrm{d}}{\mathrm{d}t}\Big|_{t=0} \, \gamma(t) \right\}.                                                         \]
It is a vector subspace whose dimension coincides with the dimension of $\Var{M}$. Moreover, at every point $x \in \Var{M}$, there exist open neighborhoods $\Var{V} \subset \Var{M}$ and $\Var{U} \subset \Tang{x}{\Var{M}}$ of $x$, and a bijective smooth map $\phi : \Var{V} \to \Var{U}$ with smooth inverse. The tuple $(\Var{V},\phi)$ is a \textit{coordinate chart}
of $\Var{M}$. A \textit{smooth map} between manifolds $F : \Var{M} \to \Var{N}$ is a map such that for every $x \in \Var{M}$ and coordinate chart $(\Var{V},\phi)$ containing $x$, and every coordinate chart $(\Var{W}, \psi)$ containing $F(x)$, we have that $\psi \circ F \circ \phi^{-1} : \phi(\Var{U}) \to \psi(F(\Var{U}))$ is a smooth map. The \textit{derivative} of $F$ can be defined as the linear map $\deriv{F}{x} : \Tang{x}{\Var{M}} \to \Tang{F(x)}{\Var{N}}$ taking the tangent vector $\vect{v} \in \Tang{x}{\Var{M}}$ to $\frac{\mathrm{d}}{\mathrm{d}t}|_{t=0} F(\gamma(t)) \in \Tang{F(x)}{\Var{N}}$ where $\gamma(t) \subset \Var{M}$ is a curve with $\gamma(0) = x$ and $\gamma'(0) = \vect{v}$.

A \textit{Riemannian manifold} $(\Var{M},g)$ is a smooth manifold $\Var{M}$ equipped with a Riemannian metric $g$, which is an inner product $g_x(\cdot, \cdot)$ on the tangent space $\Tang{x}{\Var{M}}$ that varies smoothly with $x \in \Var{M}$. If $\Var{M} \subset \R^m$, then the inherited Riemannian metric from $\R^m$ is $g_x( \vect{x},\vect{y} ) = \langle \vect{x}, \vect{y} \rangle$ for every $x \in \Var{M}$. The length of a smooth curve $\gamma:[0,1]\to\Var{M}$ is defined by
\[
\operatorname{length}_{\Var{M}}(\gamma)=\int_0^1g_{\gamma(t)}(\gamma'(t),\gamma'(t))^{1/2}\,dt,
\]
and the distance $\operatorname{dist}_{\Var{M}}(x,y)$ between two points $x,y\in\Var{M}$ is the length of the minimal curve with extremes $x$ and $y$.

In \cref{sec_introduction}, we denoted the Segre manifold of rank-1 tensors in $\R^{n_1\times \cdots \times n_d}$ by $\Var{S}$. To emphasize the format, we sometimes write $\Var{S}_{n_1,\ldots,n_d}$ instead. \Cref{sec_introduction} also defined the \textit{addition map}
\begin{equation}\label{def_Phi}
\Phi_r : \Var S^{\times r} \to \R^{n_1\times \cdots\times n_d},\quad (\tensor{A}_1,\ldots,\tensor{A}_r)\mapsto \tensor{A}_1+\cdots+\tensor{A}_r.
\end{equation}
Tensors of rank (at most) $r$ are denoted by
\begin{equation}\label{def_sigma}
  \sigma_r = \Sec{r}{\Var{S}_{n_1,\ldots,n_d}} = \Phi_r( (\Var{S}_{n_1,\ldots,n_d})^{\times r}  ) = \left\{ \sum_{i=1}^r \sten{a}{i}{1} \otimes \cdots \otimes \sten{a}{i}{d} \;|\; \sten{a}{i}{k} \in \R^{n_k} \right\}.
\end{equation}
It is a semi-algebraic set by the Tarski--Seidenberg principle \cite{BCR1998}, because it is the projection of an algebraic variety, namely the graph of $\Phi_r$ \cite{Lee2013}. Recall that this means that $\sigma_r$ can be described as the locus of points that satisfy a system of polynomial equations and inequalities; see \cite{BCR1998}.
The dimension of $\sigma_r$ equals the dimension of the smallest $\R$-variety $\overline{\sigma_r}$ containing it \cite[Chapter 2]{BCR1998}.

\subsection{Numerical analysis}
For a smooth map $f : \Var{M} \to \Var{N}$ between Riemannian manifolds $(\Var{M},g)$ and $(\Var{N},h)$ there is a standard definition of the condition number \cite{Rice1966,BC2013,BCSS1998}, which generalizes the classic case of smooth maps between Euclidean spaces, namely
\begin{align} \label{def_kappa_general}
 \kappa[f](x) = \max_{t_x \in \Tang{x}{\Var{M}}} \frac{\| (\deriv{f}{x})(t_x) \|_{\Var{N},f(x)}}{\| t_x \|_{\Var{M},x}},
\end{align}
where $\deriv{f}{x} : \Tang{x}{\Var{M}} \to \Tang{f(x)}{\Var{N}}$ is the derivative of $f$, and $\| t_x \|_{\Var{M},x} := \sqrt{ g_x( t_x, t_x) }$ for $t_x\in\Tang{x}{\Var{M}}$ (resp.~$\| t_y \|_{\Var{N},y} := \sqrt{ h_y (t_y, t_y) }$ for $t_y \in \Tang{y}{\Var{N}}$) is the norm on the tangent space $\Tang{x}{\Var{M}}$ (resp.~$\Tang{y}{\Var{N}}$) induced by the Riemannian metric $g$ (resp.~$h$).

\section{Estimating the distribution of the condition number} \label{sec_distribution}
We start by proving the second main result, \refthm{thm_probability_dist_est}, because little technical machinery is required. In the proof, we use the following identification of the condition number with the inverse of the smallest singular value of an auxiliary matrix: for $1\leq i\leq r$ let $U_i$ be a matrix whose columns form an orthonormal basis of $\Tang{\tensor{A}_i}{\Var{S}}$. Then, by \cite[Theorem 1.1]{BV2017},
\begin{equation}\label{useful_matrix}
\kappa(\tensor{A}_1,\ldots,\tensor{A}_r) = \frac{1}{\varsigma_{\min}( \deriv{\Phi_r}{(\tensor{A}_1,\ldots,\tensor{A}_r)} )} =  \frac{1}{\varsigma_{\min}(\left[\begin{smallmatrix} U_1 & \cdots & U_r \end{smallmatrix}\right])},
\end{equation}
where $\varsigma_{\min}$ denotes the smallest singular value.
The smallest singular value $\varsigma_{\min}( \deriv{\Phi_r}{(\tensor{A}_1,\ldots,\tensor{A}_r)} )$ is actually equal to the $r(n_1+\cdots+n_d-d+1)$th singular value of the Jacobian matrix of $\Phi_r$ seen as a $C^\infty$ map from $\R^{r n_1\cdots n_d}$ to $\R^{n_1 \cdots n_d}$. Moreover, from \refeqn{useful_matrix} it follows that the condition number is \emph{scale invariant}: for all $t_1,\ldots,t_r\in \R\backslash\{0\}$ we have
$
\kappa(t_1\tensor{A}_1,\ldots, t_r\tensor{A}) = \kappa(\tensor{A}_1,\ldots, \tensor{A}).
$
Cai, Fan, and Jiang \cite{CFJ13} proved tail probabilities for the maximal pairwise angle of an independent sample of uniformly distributed points on the sphere. The idea for using their results in the proof of \refthm{thm_probability_dist_est} is to lower bound the condition number by such a maximal angle. This we do next.
\begin{lemma} \label{lem_kappa_lower_nn2}
For $i=1,\ldots,r$ let $\tensor{A}_i = t_i\, \sten{a}{i}{} \otimes \sten{b}{i}{} \otimes \sten{c}{i}{} \in \R^{m_1 \times m_2 \times m_3}$ be fixed rank-$1$ tensors with~$t_i \in \R\backslash\{0\}$ and $\Vert \sten{a}{i}{}\Vert = \Vert \sten{b}{i}{}\Vert = \Vert \sten{c}{i}{}\Vert = 1$ for all $i$. Then, we have
\[
 \kappa(\tensor{A}_1, \ldots, \tensor{A}_r) \ge \max_{1 \le i\ne j \le r} \frac{1}{\sqrt{1 - |\langle \sten{c}{i}{}, \sten{c}{j}{} \rangle|}}.
\]
\end{lemma}
\begin{proof}
Without restriction we can assume that the maximum is attained for $i=1$ and $j = 2$. By~\refeqn{useful_matrix}, the condition number is the inverse of the least singular value of the matrix
\(
 T = [ U_i ]_{i=1}^r
\)
where $U_i$ is any orthonormal basis for $\Tang{\tensor{A}_i}{\Var{S}}$.
In particular, the following orthonormal bases can be chosen for $\Tang{\tensor{A}_{1}}{\Var{S}}$ and $\Tang{\tensor{A}_{2}}{\Var{S}}$ (see, e.g., \cite[Section 5.1]{BV2017}):
\begin{align*}
U_{1} &=
\begin{bmatrix}
I_{n_1} \otimes \sten{b}{1}{} \otimes \sten{c}{1}{} &
\sten{a}{1}{} \otimes Q^2_{1} \otimes \vect{c}_{1} &
\vect{a}_{1} \otimes \vect{b}_{1} \otimes Q^3_{1}
\end{bmatrix} \text{ and } \\
U_{2} &=
\begin{bmatrix}
Q^1_{2} \otimes \sten{b}{2}{} \otimes \sten{c}{2}{} &
\sten{a}{2}{} \otimes I_{n_2} \otimes \vect{c}_{2} &
\vect{a}_{2} \otimes \vect{b}_{2} \otimes Q^3_{2}
\end{bmatrix},
\end{align*}
for $Q_i^1$, $Q_i^2$, $Q_i^3$ being orthonormal bases for $\vect{a}_i^\perp$, $\vect{b}_i^\perp$, $\vect{c}_i^\perp$, respectively.
Observe that $U_1$ contains the tangent vector $\sten{a}{2}{}\otimes \sten{b}{1}{}\otimes \sten{c}{1}{}$ and $U_2$ contains the tangent vector $\sten{a}{2}{}\otimes \sten{b}{1}{}\otimes \sten{c}{2}{}$ as columns. Then, using the computation rules for inner products from \eqref{eq:innerprod}, we find that the least singular value of $T$ is smaller than
$$
\frac{\Vert\sten{a}{2}{}\otimes \sten{b}{1}{}\otimes \sten{c}{1}{}- \sten{a}{2}{}\otimes \sten{b}{1}{}\otimes \sten{c}{2}{}\Vert}{\sqrt{\Vert \sten{a}{2}{}\otimes \sten{b}{1}{}\otimes \sten{c}{1}{}\Vert^2 + \Vert \sten{a}{2}{}\otimes \sten{b}{1}{}\otimes \sten{c}{2}{}\Vert^2}}
= \frac{\sqrt{2-2\langle\sten{a}{2}{}\otimes \sten{b}{1}{}\otimes \sten{c}{1}{},\sten{a}{2}{}\otimes \sten{b}{1}{}\otimes \sten{c}{2}{}\rangle}}{\sqrt{2}}
= \sqrt{1-\langle\sten{c}{1}{},\sten{c}{2}{}\rangle}.
$$
Repeating the argument for the tangent vector $-\sten{a}{2}{}\otimes \sten{b}{1}{}\otimes \sten{c}{2}{}$ in $U_2$ we get
\[
 \kappa(\tensor{A}_1, \ldots, \tensor{A}_r) \ge \max_{1 \le i\ne j \le r} \max\left\{\frac{1}{\sqrt{1 - \langle \sten{c}{i}{}, \sten{c}{j}{} \rangle}}, \frac{1}{\sqrt{1  + \langle \sten{c}{i}{}, \sten{c}{j}{} \rangle}}\right\} = \max_{1 \le i\ne j \le r} \frac{1}{\sqrt{1 - \vert\langle \sten{c}{i}{}, \sten{c}{j}{} \rangle\vert}},
\]
concluding the proof.
\end{proof}

Now we are ready to prove \refthm{thm_probability_dist_est}.
\begin{proof}[Proof of \refthm{thm_probability_dist_est}]
Recall that for a random vector with i.i.d.\ standard normal entries $\vect{x}$, the normalized vector $\Vert \vect{x}\Vert^{-1}\,\vect{x}$ is uniformly distributed in the sphere. From the invariance of the condition number under scaling, we can assume that the entries of $\sten{c}{i}{}$, $1\leq i\leq d$, are uniformly distributed in $\mathbb{S}(\R^{m_3})$. This and \reflem{lem_kappa_lower_nn2} show that
\begin{align*}
\mathrm{P}\left[\kappa(\tensor{A}_1,\ldots,\tensor{A}_r)\geq \alpha r^{\frac{2}{m_3-1}} \right]&\geq\mathrm{P}\left[\max_{\substack{1 \le i\ne j \le r}} \frac{1}{\sqrt{1 - \vert\langle\sten{c}{i}{}, \sten{c}{j}{}\rangle\vert}}\geq \alpha r^{\frac{2}{m_3-1}} \right]\\
&= \mathrm{P}\left[r^{\frac{4}{m_3-1}}\left( {{1 - \max_{\substack{1 \le i\ne j \le r}}\vert\langle\sten{c}{i}{}, \sten{c}{j}{}\rangle\vert}}\right)\leq \alpha^{-2} \right].
\end{align*}
From \cite[Proposition 17]{CFJ13}, for any fixed $\alpha>0$, this last expression has  limit $1-e^{-K\alpha^{1-m_3}}$. This concludes the proof.
\end{proof}

\begin{figure}[tb]
\subfloat[$A$, $B$, and $C$ i.i.d.\ standard normal entries.]{\includegraphics[width=.82\textwidth]{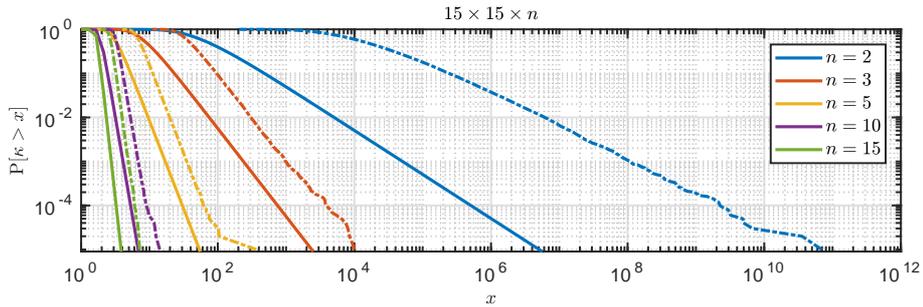}} \\
\subfloat[Arbitrary orthogonal matrices $A$ and $B$; $C$ i.i.d.\ standard normal entries.]{\includegraphics[width=.82\textwidth]{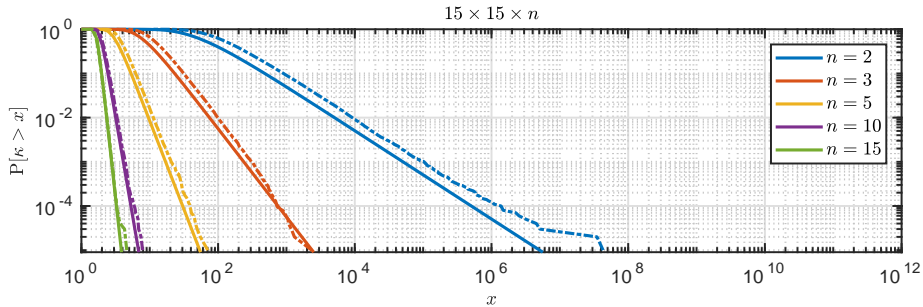}}
\caption{The empirical complementary cumulative distribution function of the condition number for rank-$15$ tensors of size $15 \times 15 \times n$ is shown in dashed lines. The corresponding solid lines show the lower bound from \cref{thm_probability_dist_est}. The tensors $\tensor{A}=\sum_{i=1}^{15} \vect{a}_i \otimes \vect{b}_i \otimes \vect{c}_i$ were generated by randomly sampling factor matrices $A\in\R^{15 \times 15}$, $B \in \R^{15\times15}$ and $C \in \R^{n \times 15}$, as indicated.}
\label{fig_cond_distribution}
\end{figure}

\Cref{thm_probability_dist_est} is illustrated in \cref{fig_cond_distribution} for $15 \times 15 \times n$ tensors of rank $15$ for $n=2,3,5,10,15$. Every solid line represents a limiting \textit{complementary cumulative distribution function} (ccdf) $\lim_{r\to\infty} T_{r,\alpha}$ from \cref{thm_probability_dist_est}, which provide asymptotic lower bounds on the ccdfs of the condition numbers of random rank-$r$ CPDs. The dashed lines in \cref{fig_cond_distribution} show the empirical ccdfs of the condition number based on two different Monte Carlo experiments.

In the first set of experiments, visualized in \cref{fig_cond_distribution}(A), we generated $10^5$ random rank-15 tensors $\tensor{A}=\sum_{i=1}^{15} \vect{a}_i \otimes \vect{b}_i \otimes \vect{c}_i$ by independently sampling the entries of the factor matrices $A = [\vect{a}_i]\in\R^{15\times 15}$, $B=[\vect{b}_i] \in \R^{15\times 15}$ and $C=[\vect{c}_i]\in\R^{n\times 15}$ from a standard normal distribution. It is observed that the limiting distribution of \cref{thm_probability_dist_est} seems to approximate the shape of the distribution of the condition numbers reasonably well. However, the lower bound seems rather weak for $n=2$. One of the main observations, which is also evident from the formula of the limiting distribution, is that \textit{as $n$ increases the probability of sampling tensors with a high condition number decreases}. As is evident from the empirical ccdf in \cref{fig_cond_distribution}(A), $n=2$ admits the worst distribution by far: there is a $10\%$ probability of sampling a condition number greater than $10^5$, and still a $0.1\%$ chance to encounter a condition number greater than~$10^8$. On the other hand, for $n=15$, all sampled tensors had a condition number less than $10$.

In the second set of experiments, shown in \cref{fig_cond_distribution}(B), we generated $10^5$ random rank-$15$ tensors of size $15 \times 15 \times n$ in a different way in order to illustrate the quality of the lower bound in \cref{thm_probability_dist_est}. This time, after sampling the factor matrices $(A,B,C)$ as above, we perform Gram--Schmidt orthogonalization of $A$ and $B$. As can be seen in \cref{fig_cond_distribution}(B), the empirical ccdfs here are close to the corresponding limiting distributions.

We had one additional reason to treat \cref{thm_probability_dist_est} first: on a fundamental level, a PBA solves the TDP for $n_1 \times n_2 \times n_3$ tensors by transforming it into a TDP for $n_1 \times n_2 \times 2$ tensors. The above experiments clearly show that the latter problem has a much worse distribution of condition numbers than the original problem. In other words, from the viewpoint of sensitivity, \textit{PBAs try to solve an easy problem via the solution of a significantly more difficult problem}. This approach is nearly guaranteed to end in instability.

\section{The manifold of $r$-nice tensors} \label{sec_nice}
While the instability of PBAs is already plausible from \cref{fig_cond_distribution}, proving \cref{thm_unstable} is substantially more complicated. In order to prove it, we should first formalize what we mean by ``solving a TDP.'' This problem is rife with subtleties.

For example, what should the solution of a TDP be if the input tensor $\tensor{A}$ is the generic rank-$11$ tensor in $\C^{11 \times 6 \times 3}$? This tensor has $352,716$ isolated CPDs \cite{HOOS2016}. Computing all of them seems computationally infeasible. Nevertheless, all of them are well-behaved because each one of these will vary smoothly in a small open neighborhood of $\tensor{A}$ in $\C^{11 \times 6 \times 3}$. On the other hand, the generic rank-$6$ tensor of multilinear rank $(4,4,4)$ $\tensor{B}$ in $\C^{6 \times 6 \times 6}$ behaves erratically. It has~$2$ isolated decompositions \cite[Theorem 1.3]{CO2012}, but a generic rank-$6$ tensor close to $\tensor{B}$ has only one decomposition that can be moved around continuously such that its limit is a decomposition of~$\tensor{B}$. This process works for both of $\tensor{B}$'s decompositions, because the rank-$6$ tensors have two smooth folds meeting in $\tensor{B}$ \cite[Example 4.2]{COV2014}. What should an algorithm compute in this case?

For an $r$-identifiable tensor $\tensor{A}$ there is an unambiguous answer to the above question. Namely, the solution is the unique set of rank-$1$ tensors $\{\tensor{A}_1,\ldots,\tensor{A}_r\}$ whose sum is $\tensor{A}$. The goal of this section is to carefully define a \emph{tensor decomposition map} $\tau_{r;n_1,\ldots,n_d}$ in \cref{Phi_inverse} whose computation solves the TDP for a subset of rank-$r$ tensors.
The domain where the smooth function $\tau_{r;n_1,\ldots,n_d}$ is well defined deserves its own definition, \cref{def_nice} below; we call it the manifold of \textit{$r$-nice tensors} $\Var N \subset \sigma_r$.
In \cref{prop_image} we prove that $\Var N$ is a Zariski open dense subset of the set of rank-$r$ tensors, so that ``almost all tensors are $r$-nice.''

Before defining $\Var{N}$, we first need the following two standard definitions.
If for a collection of $r$ vectors $\vect p_1,\ldots,\vect p_r \in \R^n$ every subset of $\min\{r,n\}$ many vectors is linearly independent, then the vectors are said to lie in \textit{general linear position} (GLP). We say that a collection of $r$ rank-$1$ tensors $\{ \sten{a}{i}{1} \otimes \cdots \otimes \sten{a}{i}{d} \}_i$ is in \textit{super general linear position} (SGLP) if for every $1\leq s\leq d$ and every $\vect{h} \subset \{1,\ldots,d\}$ with $|\vect{h}| = s$, the set $\{ \sten{a}{i}{h_1} \otimes \cdots \otimes \sten{a}{i}{h_s} \}_i$ is in GLP.

\begin{dfn}[$r$-nice tensors]\label{def_nice}
Recall from~\refeqn{def_sigma} the definition of rank-$r$ tensors $\sigma_r$ and its closure $\overline{\sigma_r}$. Then, $\Var{M}_{r;n_1,\ldots,n_d}\subset \Var{S}_{n_1,\ldots,n_d}^{\times r}$ is defined to be the set containing all the rank-$1$ tuples $\tuple{a}=(\tensor{A}_1,\ldots,\tensor{A}_r)$ satisfying the following properties:
\begin{enumerate}
\item[(i)] $\Phi_r(\tuple{a})$ is a smooth point of $\overline{\sigma_r}$,
\item[(ii)] $\Phi_r(\tuple{a})$ is $r$-identifiable, and, thus, has rank equal to $r$,
\item[(iii)] $\tuple{a}$ has finite condition number,
\item[(iv)] $\tuple{a}$ is in SGLP, and
\item[(v)] for all $i$ the $(1,1,\ldots,1)$-entry of $\tensor{A}_i$ is not equal to zero.
\end{enumerate}
The $r$-nice tensors $\Var{N}_{r;n_1,\ldots,n_d}$ are defined to be the image of $\Var{M}_{r;n_1,\ldots,n_d}$ under the addition map~$\Phi_r$ from \refeqn{def_Phi}:
\[
\Var{N}_{r;n_1,\ldots,n_d} := \Phi_r( \Var{M}_{r;n_1,\ldots,n_d} ).
\]
If it is clear from the context we drop the subscript from both $\Var{M}_{r;n_1,\ldots,n_d}$ and $\Var{N}_{r;n_1,\ldots,n_d}$ and simply write $\Var M$ and $\Var N$.
\end{dfn}
\begin{rem}
The reason for the last requirement, (v), is that under this restriction we can define a parametrization of rank-$1$ tensors that is a \textit{diffeomorphism}; see the next subsection for details.
\end{rem}

\subsection{Elementary results} \label{sec_diff_geom}
Before proceeding, we need a few elementary results related to the differential geometry of CPDs, which we did not find in the literature.

The rank-$1$ tensors in $\R^{n_1 \times \cdots \times n_d}$, i.e., 
$
\Var{S} = \{ \vect{a}_1 \otimes \cdots \otimes \vect{a}_d \;|\; \vect{a}_k \in \R^{n_k} \}\setminus\{0\},
$
form the affine cone over a smooth projective variety (see, e.g., \cite{Landsberg2012}) and, hence, $\Var{S}$ is an analytic submanifold of $\R^{n_1 \times \cdots \times n_d}$. Its dimension is $1 + \sum_{k=1}^d (n_k-1)$ \cite{Landsberg2012}. The map\footnote{The following items are most naturally considered in projective space, but in order to avoid as much technicalities as is feasible we prefer to present the results concretely as subspaces of Euclidean spaces.}
\[
 \Psi_{n_1,\ldots,n_d}: \R\setminus\{0\} \times \mathbb{S}(\R^{n_1}) \times \cdots \times \mathbb{S}(\R^{n_d}) \to \Var{S},\; (\lambda, \vect{u}_1, \ldots, \vect{u}_d) \mapsto \lambda \vect{u}_1 \otimes \cdots \otimes \vect{u}_d
\]
is a surjective \emph{local diffeomorphism}: every point in the domain has an open neighborhood such that $\Psi_{n_1,\ldots,n_d}$ restricted to this neighborhood is an open, smooth ($C^\infty$), bijective map with smooth inverse \cite[p.~79]{Lee2013}. Indeed, it can be verified that the derivative is injective at every point; see, e.g., \cite[Section 5.1]{BV2017}. Note that the fiber of $\Psi_{n_1,\ldots,n_d}$ at $\lambda \vect{u}_1 \otimes \cdots \otimes \vect{u}_d$ is exactly the set
\(
 \{ (\omega_0 \lambda, \omega_1 \vect{u}_1, \ldots, \omega_d \vect{u}_d) \;|\;  \omega_0 \cdots\omega_d = 1, \omega_i \in \{ -1, 1 \} \}
\), which has $2^{d}$ elements.
Moreover, $\Psi_{n_1,\ldots,n_d}$ is a proper map so that it is a $2^{d}$-sheeted smooth covering map \cite[p.~91--95]{Lee2013}.

Let $\mathbb{S}^+(\R^n) = \{ \vect{u} \in \mathbb{S}(\R^n) \;|\; u_1 > 0 \}$ be the ``upper'' half of the unit sphere; it is a submanifold in the subspace topology on $\R^n$. Let us define the following restriction of $\Psi$:
\begin{equation}\label{eqn_psi_star}
\Psi^*_{n_1,\ldots,n_d}: \R\setminus\{0\} \times \mathbb{S}^+(\R^{n_1}) \times \cdots \times \mathbb{S}^+(\R^{n_d}) \to \Var S, \;(\lambda, \vect{u}_1, \ldots, \vect{u}_d) \mapsto \lambda \vect{u}_1 \otimes \cdots \otimes \vect{u}_d.
\end{equation}
It follows from the foregoing that $\Psi^*_{n_1,\ldots,n_r}$ is a bijective local diffeomorphism onto its image, so it is a \textit{{(global)} diffeomorphism} onto its image. Let $\Var{S}^*_{n_1,\ldots,n_r}$ be the image of $\Psi^*_{n_1,\ldots,n_r}$:
\begin{equation} \label{def_phi_star}
\Var{S}^*_{n_1,\ldots,n_r} := \Psi^*_{n_1,\ldots,n_r}(\R\setminus\{0\} \times \mathbb{S}^+(\R^{n_1}) \times \cdots \times \mathbb{S}^+(\R^{n_d})).
\end{equation}
When it is clear from the context we drop the subscripts from $\Psi_{n_1,\ldots,n_d}$, $\Psi_{n_1,\ldots,n_d}^*$ and $\Var{S}_{n_1,\ldots,n_d}^*$. The foregoing explains part (v) in \cref{def_nice}: we wish to work with a parametrization of $\Var S$ that is a diffeomorphism, so we restrict ourselves to $\Var S^*$ and use $\Psi^*$. We will show in the proof of \cref{prop_nice} that $\Var{S}^*$ is open in the Zariski topology and, hence, open and dense in the Euclidean topology.

Finally, we consider the subset $S_{r;n}\subset (\mathbb{S}^+(\R^n))^{\times r}$ defined as
\begin{equation}\label{def_S}
S_{r;n} = \{(s_1,\ldots, s_r)\in  (\mathbb{S}^+(\R^n))^{\times r}\mid [s_1,\ldots,s_r]\in \R^{n\times r} \text{ has full rank}\}.
\end{equation}
Note that $S_{r;n}$ is an open submanifold, because the locus of points not satisfying the rank condition is closed in the Zariski topology. We also have the following result.
\begin{lemma}\label{lemma_quotient1}
Let $\mathfrak{S}_r$ be the symmetric group on $r$ elements. Then, $\widehat{S}_{r;n} := S_{r;n}/\mathfrak{S}_r$ is a manifold. {Moreover, the projection $\pi : S_{r;n} \to \widehat{S}_{r;n}, (x_1,\ldots,x_r) \mapsto \{x_1, \ldots, x_r\}$ is a local diffeomorphism.}
\end{lemma}
\begin{proof}
$\mathfrak{S}_r$ is a discrete Lie group acting smoothly \cite[Example 7.22(e)]{Lee2013}. The group action is also free because $S \in S_{r;n}$ can be a fixed point of some permutation only if $s_i, s_j \in S$ with $i \ne j$ are equal. It can be verified that the conditions in \cite[Lemma 21.11]{Lee2013} hold, so that the action is proper. The result follows by the quotient manifold theorem \cite[Theorem 21.10]{Lee2013}.
\end{proof}

\subsection{Differential geometry of $r$-nice tensors} \label{sec_sec_diff_geom_nice}
Recall that a Segre manifold $\Var{S} \subset \R^{n_1 \times \cdots \times n_d}$ is said to be \emph{generically $r$-identifiable} if all tensors in a Zariski-open subset of
\(
\overline{\sigma_r}
\)
are identifiable; see \cite{COV2014,COV2017} for the state of the art.
In the context of PBAs, the following standard result suffices.

\begin{lemma} \label{lem_identifiable}
Let $n_1 \ge n_2 \ge \cdots \ge n_d \ge 2$. If $r \le n_2$, then $\Var{S}_{n_1,\ldots,n_d}$ is generically $r$-identifiable.
\end{lemma}
\begin{proof}
This is follows, for example, from the effectiveness of Kruskal's criterion; see \cite{COV2017}.
\end{proof}

Next, we prove an important property of the set $\Var M_{r;n_1,\ldots,n_d}$ from \cref{def_nice}.

\begin{prop} \label{prop_nice}
Let $\Var{S}_{n_1,\ldots,n_d}$ be generically $r$-identifiable. Then,
$\Var{M}_{r;n_1,\ldots,n_d}$ is a Zariski-open submanifold of $ \Var{S}_{n_1,\ldots,n_d}^{\times r}$.
\end{prop}
\begin{proof}
Let $\Var{S} = \Var{S}_{n_1,\ldots,n_d}$ and $\Var{M} = \Var{M}_{r;n_1,\ldots,n_d}$ for brevity. We show that the set of tuples \textit{not} satisfying either of the conditions in \cref{def_nice} is contained in a union of five Zariski-closed subsets of $\Var{M}$; these subsets are denoted by $\Var B_{\mathrm{(i)}}, \Var B_{\mathrm{(ii)}}, \Var B_{\mathrm{(iii)}}, \Var B_{\mathrm{(iv)}}$ and $\Var B_{\mathrm{(v)}}$. Taking
$$
\Var{M}= \Var S^{\times r} \backslash \left(\Var B_{\mathrm{(i)}} \cup \Var B_{\mathrm{(ii)}} \cup \Var B_{\mathrm{(iii)}} \cup \Var B_{\mathrm{(iv)}} \cup \Var B_{\mathrm{(v)}}\right)
$$
would then prove the assertion.

Recall that generic $r$-identifiability implies \textit{nondefectivity} of $\sigma_r$; see \cite[Chapter 5, specifically Corollary 5.3.1.3]{Landsberg2012}. Hence, $\dim \sigma_r = \dim \overline{\sigma_r} = \dim \Var{S}^{\times r} = r \dim \Var{S}$. The subvariety $\Sigma \subset \overline{\sigma_r}$ of singular points is proper and closed in the Zariski topology by definition \cite{Harris1992}.
This means that in addition to the polynomials that vanish on the $\R$-variety $\overline{\sigma_r}$, there are $k \ge 1$ additional nontrivial polynomial equations with coefficients over $\R$ such that $f_1(y) = \cdots = f_k(y) = 0$ for all $y \in \Sigma$. If $y$ has a preimage $x\in\Var{S}^{\times r}$ under $\Phi_r$, then $f_1(\Phi_r(x)) = \cdots = f_k(\Phi_r(x)) = 0$. Hence, the locus $\Var{B}_{\mathrm{(i)}}$ of decompositions not satisfying condition (i) in \cref{def_nice}, which map into the singular locus $\Sigma$ under $\Phi_r$ is a Zariski-closed set. It is also a proper subset, because otherwise $\Phi_r(\Var{S}^{\times r}) = \sigma_r \subset \Sigma$, which is a contradiction as $\dim \Sigma < \dim \overline{\sigma_r} = \dim \sigma_r$.

The set of tensors in $\overline{\sigma_r}$ with several decompositions is closed in the Zariski topology by assumption. We can apply the same argument as in the previous paragraph to conclude that the variety of decompositions $\Var{B}_{\mathrm{(ii)}} \subset \Var{S}^{\times r}$ that map to points of $\overline{\sigma_r}$ that are not $r$-identifiable is a proper Zariski closed subset in $\Var{S}^{\times r}$.

The subset $\Var{B}_{\mathrm{(iii)}} \subset \Var{S}^{\times r}$ of decompositions with condition number $\infty$, is contained in a Zariski-closed set if the $r$-secant variety $\overline{\sigma_r}$ is nondefective by \cite[Lemma 5.3]{BV2018b}.

The set of points $\Var{B}_{\mathrm{(iv)}} \subset \Var{S}^{\times r}$  not satisfying (iv) is Zariski-closed by \cite[Lemma 4.4]{COV2017}.

For the last point, observe that condition (v) of \cref{def_nice} is equivalent to $p\in (\Var S^*)^{\times r}$. By definition of $\Var{S}^*$ in \cref{def_phi_star}, the set of points in $\Var{S}\setminus\Var{S}^*$ is the intersection of $\Var{S}$ with the union of the following linear varieties:
\(
 L_k = \R^{n_1} \otimes \cdots \otimes \R^{n_{k-1}} \otimes \R^{n_k}/\langle \vect{e}_1 \rangle \otimes \R^{n_{k+1}} \otimes \cdots \otimes \R^{n_d},
\)
where $\R^{n_k}/\langle \vect{e}_1\rangle = \langle \vect{e}_2, \ldots, \vect{e}_{n_k} \rangle$ and $\vect{e}_i$ is the $i$th standard basis vector of $\R^{n_k}$. In fact,
\[{
 \Var{S}\setminus\Var{S}^* = \Var{S} \cap \left( \bigcup_{k=1}^d L_k \right) = \bigcup_{k=1}^d ( \Var{S} \cap L_k ) \simeq \bigcup_{k=1}^d \Var{S}_{n_1,\ldots,n_{k-1},n_k-1,n_{k+1},\ldots,n_d},
}\]
which is thus a Zariski-closed set {because $\dim \Var{S}_{n_1,\ldots,n_{k-1},n_k-1,n_{k+1},\ldots,n_d} < \dim \Var{S}$}. Therefore, taking
\(
\Var{B}_{\mathrm{(v)}} = \bigcup_{i=1}^r \Var{S}^{\times (i-1)} \times (\Var{S}\setminus\Var{S}^*) \times \Var{S}^{\times (r-i)}
\)
yields the Zariski-closed variety of points not satisfying (v). This concludes the proof.
\end{proof}
The definition of $\Var M_{r;n_1,\ldots,n_d}$ is nice, because the addition map $\Phi_r$ from \refeqn{def_Phi} restricted to $\Var M_{r;n_1,\ldots,n_d}$ is a local diffeomorphism. However, we wish to work with global diffeomorphisms and therefore need the following proposition.

\begin{prop} \label{prop_permutation}
If $\Var{S}_{n_1,\ldots,n_d}$ is generically $r$-identifiable, then $\widehat{\Var{M}}_{r;n_1,\ldots,n_d} = \Var{M}_{r;n_1,\ldots,n_d}/\mathfrak{S}_r$ is a manifold. Moreover, the projection $\widehat{\pi} : \Var{M}_{r;n_1,\ldots,n_d} \to \widehat{\Var{M}}_{r;n_1,\ldots,n_d}, (\tensor{A}_1,\ldots,\tensor{A}_r) \mapsto \{\tensor{A}_1, \ldots, \tensor{A}_r\}$ is a local diffeomorphism.
\end{prop}
\begin{proof}
Combine the proof of \cref{lemma_quotient1} with the fact that $r$-identifiability implies that the rank-1 tensors in a decomposition $(\tensor{A}_1,\ldots,\tensor{A}_r)\in \Var{M}_{r;n_1,\ldots,n_d}$ are pairwise distinct.
\end{proof}

It is clear that the addition map $\Phi_r$ is constant on $\mathfrak{S}_r$-orbits in $\Var{M}_{r;n_1,\ldots,n_d}$. Therefore, $\Phi_r$ is well defined on $\widehat{\Var{M}}_{r;n_1,\ldots,n_d}$.
Now, we have the following crucial result.

\begin{prop} \label{prop_image}
Let $\Var{N}_{r}^{n_1,\ldots,n_d}\subset \sigma_r$ be the set of $r$-nice tensors. If $\Var{S}_{n_1,\ldots,n_d}$ is generically $r$-identifiable, then
\[
\Phi_r: \widehat{\Var{M}}_{r;n_1,\ldots,n_d} \to \Var{N}_{r}^{n_1,\ldots,n_d},\; \{\tensor{A}_1,\ldots,\tensor{A}_r\} \mapsto  \tensor{A}_1+\cdots +\tensor{A}_r
\]
is a diffeomorphism.
Moreover, $\Var{N}_{r;n_1,\ldots,n_d}$ is an open dense submanifold of $\sigma_r$.
\end{prop}
\begin{proof}
As before, for brevity, we drop all subscripts. Let $\tuple{a}=(\tensor{A}_1,\ldots,\tensor{A}_r)\in \Var{M}$. By definition, $\tuple{a}$ has a finite condition number. This means, by \cite[Theorem 1.1]{BV2017}, that the derivative of $\Phi_r$ at $\tuple{a}$ is injective. Hence, $\Phi_r$ is a smooth immersion \cite[p. 78]{Lee2013}. By the generic $r$-identifiability assumption, it follows that the $r$-secant variety $\overline{\sigma_r}$ is not defective so that $\dim \overline{\sigma_r} = r \dim \Var{S}$. Moreover, by \cref{prop_nice}, we have $r \dim \Var{S} = \dim \Var{M}_{r;n_1,\ldots,n_d}$ and, by construction, we have $\dim \Var{M}_{r;n_1,\ldots,n_d} = \dim \widehat{\Var{M}}_{r;n_1,\ldots,n_d}$. As $\Phi_r$ is injective by generic $r$-identifiability and by having taken the particular quotient in \cref{prop_permutation}, then \cite[Proposition 4.22(d)]{Lee2013} entails that $\Phi_r$ is a smooth embedding. The first conclusion follows by \cite[Proposition 5.2]{Lee2013}.

The foregoing already shows that $\Var{N}_{r;n_1,\ldots,n_d} \subset \sigma_r$ is open. We show that it is dense. Let $\tensor{A} \in \sigma_r \setminus \Var{N}_{r;n_1,\ldots,n_d}$ with decomposition $\tensor{A} = \Phi_r(\tuple{a})= \tensor{A}_1 + \cdots + \tensor{A}_r$. By \cref{prop_nice}, there exist a sequence
\[
 ( \tensor{A}_1^{(j)}, \ldots, \tensor{A}_r^{(j)} ) \in \Var{M}_{r;n_1,\ldots,n_d} \text{ such that } \lim_{j \to \infty}  ( \tensor{A}_1^{(j)}, \ldots, \tensor{A}_r^{(j)} ) \to ( \tensor{A}_1, \ldots, \tensor{A}_r ).
\]
Note that this is convergence in the usual Euclidean topology that $\Var{M}_{r;n_1,\ldots,n_d}$ inherits from the ambient space $(\R^{n_1 \times \cdots \times n_d})^{\times r}$. Consequently, the components also converge individually: $
 \lim_{j\to\infty} \tensor{A}_i^{(j)} \to \tensor{A}_i$, $i = 1, \ldots, r.
$
The result follows from the fact that adding the above convergent sequences results in a convergent sequence in $\Var{N}_{r;n_1,\ldots,n_d}$ with limit $\tensor{A}$. Hence, $\tensor{A} \in \overline{ \Var{N}_{r;n_1,\ldots,n_d} }$ so that $\Var{N}_{r;n_1,\ldots,n_d}$ is dense in $\sigma_r$, concluding the proof.
\end{proof}

From Proposition \ref{prop_image}, $\Phi_r$ has an smooth inverse, which solves the TDP on $\Var{N}_{r;n_1,\ldots,n_d} \subset \R^{n_1 \times \cdots \times n_d}$. We finally arrive at the goal of this section.
\begin{dfn}\label{Phi_inverse}
The inverse of $\Phi_r$ on the manifold of $r$-nice tensors is
\begin{equation*}
 \tau_{r;n_1,\ldots,n_d} : \Var{N}_{r;n_1,\ldots,n_d} \to \widehat{\Var{M}}_{r;n_1,\ldots,n_d},\; \tensor{A}_1+\cdots+\tensor{A}_r \mapsto \{ \tensor{A}_1, \ldots, \tensor{A}_r \}.
\end{equation*}
We call this mapping the \emph{tensor decomposition map}.
\end{dfn}

\begin{rem}\label{rem_local_identification}
One way to interpret the construction in this section is that near $\tensor{A} \in \Var{N}_{r;n_1,\ldots,n_d}$ we locally have the identification
\(
 \tau_{r;n_1,\ldots,n_d} = \widehat{\pi} \circ \Phi_{\tuple{a}}^{-1},
\)
where $\tuple{a}=(\tensor{A}_1,\ldots,\tensor{A}_r)$ is any ordered $r$-nice decomposition of $\tensor{A}$, $\Phi_{\tuple{a}}^{-1}$ is the local inverse in \cref{def_kappa}, and $\widehat{\pi}$ is as in \cref{prop_permutation}.
\end{rem}

\subsection{Implications for the condition number} \label{sec_implications_cn}
Let $\tuple{a} = (\tensor{A}_1,\ldots,\tensor{A}_r)$ be any ordered $r$-nice decomposition in $\Var{M}_{r;n_1,\ldots,n_d}$.
For the $r$-nice tensor $\tensor{A} = \tensor{A}_1 + \cdots + \tensor{A}_r \in \Var{N}_{r;n_1,\ldots,n_d}$, we will relate the condition number $\kappa[\tau_{r; n_1, \ldots, n_d}](\tensor{A})$, as in \cref{def_kappa_general}, to the condition number of the CPD $\kappa(\tensor{A}_1,\ldots,\tensor{A}_r)$ from \cite{BV2017}. We have the following result.

\begin{lemma} \label{lem_kappa_identity}
Let us choose the Riemannian metrics on $\Var{N}_{r;n_1,\ldots,n_d}$ and $\Var{M}_{r;n_1,\ldots,n_d}$ inherited from their respective ambient spaces. Then, the mapping $\widehat{\pi}$ from \cref{prop_permutation} induces a natural Riemannian metric on $\widehat{\Var{M}}_{r;n_1,\ldots,n_d}$ with the following properties:
\begin{enumerate}
  \item $\widehat{\pi}$ is a local isometry;
  \item for all $\tensor{A} = \tensor{A}_1 + \cdots + \tensor{A}_r \in \Var{N}_{r;n_1,\ldots,n_d}$, we have
  \(
  \kappa(\tensor{A}_1,\ldots,\tensor{A}_r)
  = \kappa[\tau_{r;n_1,\ldots,n_d}](\tensor{A})
  \); and
  \item for any $\{\tensor{A}_1, \ldots, \tensor{A}_r\},\{\tensor{B}_1, \ldots, \tensor{B}_r\}\in \widehat{\Var{M}}$ we have
  \[
  \mathrm{dist}_{\widehat{\Var{M}}}(\{\tensor{A}_1, \ldots, \tensor{A}_r\},\{\tensor{B}_1, \ldots, \tensor{B}_r\})=\min_{\pi \in \mathfrak{S}_r}\left(	\mathrm{dist}_{\Var{M}}((\tensor{A}_1,\ldots,\tensor{A}_r),\pi(\tensor{B}_1,\ldots,\tensor{B}_r))\right).
  \]
\end{enumerate}
Here, $\mathrm{dist}_{\widehat{\Var{M}}}$ and $\mathrm{dist}_{\Var{M}}$ are the respective Riemannian distances.
\end{lemma}
Because of the equality of condition numbers in \cref{lem_kappa_identity} and \cref{def_kappa}, we find that for every $\tuple{a}=(\tensor{A}_1,\ldots,\tensor{A}_r) \in \Var{M}_{r;n_1,\ldots,n_d}$ we have
\[
\kappa[\tau](\tensor{A})
= \lim_{\epsilon\to 0} \sup_{\substack{\tensor{B} \in \sigma_r\\ \|\tensor{A}-\tensor{B}\|_F \le \epsilon}} \frac{\mathrm{dist}_{\widehat{\Var{M}}}(\tau(\tensor{A}),\tau(\tensor{B}))
	}{\|\tensor{A}-\tensor{B}\|_F}
= \lim_{\epsilon\to 0} \sup_{\substack{\tensor{B} \in \sigma_r\\ \|\tensor{A}-\tensor{B}\|_F \le \epsilon}} \min_{\pi \in \mathfrak{S}_r} \frac{\| \Phi_{\tuple{a}}^{-1}(\tensor{A}) - \pi \circ \Phi_{\tuple{a}}^{-1}(\tensor{B}) \|_F}{\|\tensor{A}-\tensor{B}\|_F},
\]
where $\tau=\tau_{r;n_1,\ldots,n_d}$ and the last equality follows from $(3)$ in Lemma \ref{lem_kappa_identity}.
This above equality is very significant because it allows us to make sense of the distance between two \textit{unordered} CPDs, i.e., sets of rank-$1$ tensors, $\{ \tensor{A}_1, \ldots, \tensor{A}_r \}$ and $\{ \tensor{B}_1, \ldots, \tensor{B}_r \}$. %
As a consequence, we get an instance of the well-known rule of thumb in numerical analysis:
\begin{align} \label{eqn_ruleofthumb}
\underbrace{\min_{\pi\in\mathfrak{S}_r} \| A - B P_{\pi} \|_F}_{\text{forward error}} \lesssim \underbrace{\kappa[\tau](\tensor{A})}_{\text{condition number}} \cdot \underbrace{\| \tensor{A} - \tensor{B} \|_F}_{\text{backward error}}
\end{align}
for nearby $\tensor{A} = \tensor{A}_1 + \cdots + \tensor{A}_r$ and $\tensor{B} = \tensor{B}_1 + \cdots + \tensor{B}_r$; herein, $A = [ \tensor{A}_i ]_i \in \R^{n_1\cdots n_d \times r}$
(resp. $B = [ \tensor{B}_i ]_i \in \R^{n_1\cdots n_d \times r}$) is a matrix that contains the vectorized rank-$1$ tensors $\tensor{A}_i$ (resp.~$\tensor{B}_i$) as columns, and $P_\pi$ is the $r \times r$ permutation matrix representing the permutation $\pi$. The notation~$\lesssim$ indicates that the bound is asymptotically sharp for infinitesimal $\|\tensor{A}-\tensor{B}\|_F$.

\section{Pencil-based algorithms for the CPD} \label{sec_pencilbased}
We start by specifying a very general class of numerical algorithms to which the analysis in \cref{sec_unstable_algorithm} applies. The construction may seem a bit abstract at first sight, so it is useful to keep in mind that the prototypical algorithm from the introduction is an example of a PBA.

As it suffices, in principle, to present a single input for which an algorithm is unstable, we can choose a well-behaved subset of $r$-nice tensors $\Var{N}^* \subset \Var{N}_{r;n_1,n_2,n_3} \subset \R^{n_1 \times n_2 \times n_3}$ (for the exact choice of $\Var N^*$ see \cref{def_pencil_algorithm} below) and specify what a PBA should compute for such inputs. If the numerical instability already occurs on this subset, then it is also unstable on larger domains.
We recall from \cref{sec_nice} that by considering only $r$-nice tensors $\Var{N}_{r;n_1,n_2,n_3}$, the TDP consists of computing the action of the function $\tau_{r;n_1,n_2,n_3}$ from \cref{Phi_inverse}. PBAs compute this map in a particular way, via the four transformations described below.

The input of a PBA is assumed to be the multidimensional array $\tensor{A} \in \R^{n_1 \times n_2 \times n_3}$. The first transformation is the multilinear multiplication $\rho_Q$ that maps $n_1\times n_2\times n_3$ tensors to format $n_1\times n_2\times 2$ via the matrix $Q \in \R^{n_3 \times 2}$ with orthonormal columns:
\[
 \rho_Q: \R^{n_1 \times n_2 \times n_3} \to \R^{n_1 \times n_2 \times 2},\; \tensor{A} \mapsto (I_{n_1}, I_{n_2}, Q^T) \cdot \tensor{A}.
\]

The second transformation, $\widehat{\theta}$, computes the set of unit-norm columns of the first factor matrix $A$ of the CPD when restricted to $\Var{N}_{r;n_1,n_2,2}$:
\begin{equation*} 
 \widehat{\theta}|_{\Var{N}_{r;n_1,n_2,2}} : \Var{N}_{r;n_1,n_2,2} \to \widehat{S}_{r;n_1} ,\quad\tensor{B} =
 \sum_{i=1}^r \sten{a}{i}{} \otimes \sten{b}{i}{} \otimes {\vect{z}}_i \mapsto \{ \vect{a}_1, \ldots, \vect{a}_r\}.
\end{equation*}
Herein, $\widehat{S}_{r;n_1} = S_{r;n_1}/\mathfrak{S}_r$, where $S_{r;n_1}$ is as in \cref{def_S}.
Note the curious definition of $\widehat{\theta}$ involving the restriction to $\Var{N}_{r;n_1,n_2,n_3}$. The reason for this formulation is that a PBA will be executed using floating-point arithmetic. It is unlikely that the floating point representation $\operatorname{fl}(\tensor{B}) \in \R^{n_1 \times n_2 \times 2}$ is exactly in $\Var{N}_{r;n_1,n_2,2} \subset \R^{n_1 \times n_2 \times 2}$, even when $\tensor{B} \in \Var N_{r;n_1,n_2,2}$. Therefore, a minimal additional demand is placed on $\widehat{\theta}$: For every $\tensor{B} \in \Var{N}_{r;n_1,n_2,2}$, $\widehat{\theta}$ must be defined for $\operatorname{fl}(\tensor{B})$.

The third transformation, $\upsilon$, when restricted to
\[
\Var{R}_{r;n_1,n_2,n_3} := \bigl\{ (\tensor{A}, A) \mid A = (\vect{a}_1,\ldots,\vect{a}_r)\in S_{r;n_1} \text{ and } \tensor{A}=\sum_{i=1}^r \vect{a}_i \otimes \vect{b}_i \otimes \vect{c}_i \in \Var{N}_{r;n_1,n_2,n_3} \bigr\},
\]
essentially computes the Khatri--Rao product $B\odot C$ of the remaining factor matrices, namely
\[
 \upsilon|_{\Var{R}_{r;n_1,n_2,n_3}} : \Var{R}_{r;n_1,n_2,n_3} \to \Var{S}_{n_2,n_3}^{\times r}, \; \Bigl(\tensor{A}=\sum_{i=1}^r \vect{a}_i \otimes \vect{b}_i \otimes \vect{c}_i, (\vect{a}_1,\ldots,\vect{a}_r) \Bigr) \mapsto ( \vect{b}_1 \otimes \vect{c}_1, \ldots, \vect{b}_r \otimes \vect{c}_r ).
\]
For the proof of instability in \cref{sec_3o_unstable}, it will not matter if or how $\upsilon$ is defined outside of $\Var{R}_{r;n_1,n_2,n_3}$, so we impose no further constraints.
The final step computes the (unordered) Khatri--Rao product of two ordered sets of vectors:
\begin{align*}
\widehat{\odot} : \R^{p \times r} \times \R^{q \times r} \to \Var{S}_{p,q}^{\times r} / \mathfrak{S}_r,\;
\bigl( (\vect{x}_1,\ldots,\vect{x}_r), (\vect{y}_1, \ldots, \vect{y}_r \bigr) &\mapsto \{ \vect{x}_1 \otimes \vect{y}_1, \ldots, \vect{x}_r \otimes \vect{y}_r \}.
\end{align*}
Applied to $A$ and $B \odot C$, this yields the set of rank-$1$ tensors solving the TDP.

We will define a PBA to be an algorithm composing the above functions. The input space for a PBA is thus $\Var{N}^* := \rho_Q^{-1}(\Var{N}_{r;n_1,n_2,2}) \cap \Var{N}_{r;n_1,n_2,n_3}.$ (it is the subset $\Var{N}^*$ mentioned at the start of this section).
Hence, we arrive at the definition of the class of PBAs for solving a TDP whose input is in $\Var{N}^* \subset \R^{n_1 \times n_2 \times n_3}$.

\begin{dfn}[Pencil-based algorithm] \label{def_pencil_algorithm}
 A pencil-based algorithm for solving the TDP is an algorithm that computes the tensor decomposition map $\tau_{r;n_1,n_2,n_3}$ when given the $n_1 \times n_2 \times n_3$ input array
 \(
\tensor{A} = \sum_{i=1}^r  \vect{a}_i \otimes \vect{b}_i \otimes \vect{c}_i \in \Var{N}^*,
 \)
where $\vect{a}_i \in \mathbb{S}^+(\R^{n_1})$ and $\Var{N}^* = \rho_Q^{-1}(\Var{N}_{r;n_1,n_2,2}) \cap \Var{N}_{r;n_1,n_2,n_3}$, by performing the following steps:
 \begin{enumerate}
  \item[S1.\phantom{0}] $\tensor{B} \leftarrow \rho_Q(\tensor{A})$;
  \item[S2.\phantom{0}] $\{ \vect{a}_1 , \ldots, \vect{a}_r\} \leftarrow \widehat{\theta}(\tensor{B})$;
  \item[S3.a] Choose an order $A := (\vect{a}_1 , \ldots, \vect{a}_r)$;
  \item[S3.b] $( \vect{b}_1 \otimes \vect{c}_1, \ldots, \vect{b}_r \otimes \vect{c}_r ) \leftarrow \upsilon(\tensor{A}, A)$;
  \item[S4.\phantom{0}] output $\leftarrow \widehat{\odot}\bigl( (\vect{a}_1 , \ldots, \vect{a}_r), ( \vect{b}_1 \otimes \vect{c}_1, \ldots, \vect{b}_r \otimes \vect{c}_r ) \bigr)$.
 \end{enumerate}
\end{dfn}

\section{Pencil-based algorithms are unstable} \label{sec_unstable_algorithm} \label{sec_3o_unstable}
We continue by showing that PBAs are numerically forward unstable for solving the TDP for third-order tensors. For $\tensor{A}\in\Var{N}^* \subset \R^{n_1\times n_2\times n_3}$ let $\{ \widetilde{\tensor{A}}_1, \ldots, \widetilde{\tensor{A}}_r \}$ be the CPD returned by a PBA in floating-point representation. The overall goal in the proof of \cref{thm_unstable} is showing that for all small $\epsilon > 0$ there exists an open neighborhood $\mathcal{O}_\epsilon \subset \Var{N}_{r;n_1,n_2,n_3}$ of $r$-nice tensors such that for $\tensor{A}=\tensor{A}_1+\cdots+\tensor{A}_r$ in that neighborhood the \textit{excess factor}
\begin{equation}\label{def_mu}
\omega(\tensor{A})
:= \frac{\min_{\pi\in\mathfrak{S}_r}\sqrt{\sum_{i=1}^r \Vert \tensor{A}_i-\widetilde{\tensor{A}}_{\pi(i)}\Vert^2_F}}{\kappa[\tau_{r; n_1,n_2,n_3}](\tensor{A}) \cdot \| \tensor{A} - \operatorname{fl}(\tensor{A}) \|_F }
\end{equation}
is at least a constant times $\epsilon^{-1}$. The exact statement is in \cref{thm_neighborhood} below.

We call $\omega$ the excess factor because it measures by how much the forward error\footnote{Recall its definition from \cref{eqn_ruleofthumb}.} produced by the numerical algorithm, as measured by the numerator, exceeds the forward error that one can expect from solving the TDP (which is equivalent to computing the map $\tau_{r; n_1,n_2,n_3}$), as measured by the denominator. Showing that the excess factor can become arbitrarily large on the domain of $\tau_{r; n_1,n_2,n_3}$ is essentially equivalent to the standard definition of numerical forward instability of an algorithm for computing $\tau_{r; n_1,n_2,n_3}$ \cite{higham}. In fact, the excess factor can be interpreted as a quantitative measure of the forward numerical instability of an algorithm on a particular input. Ideally, $\omega$ is bounded by a small constant, but for numerically unstable algorithms $\omega$ is ``too large'' relative to the problem dimensions. The next result is a more precise version of \cref{thm_unstable} which states that for all $\tensor{A} \in \Var{O}_\epsilon$, \textit{a PBA becomes arbitrarily unstable} as $\epsilon \to 0$, irrespective of the problem size.

\begin{thm}\label{thm_neighborhood}
There exist a constant $k>0$ and a tensor $\tensor{O}\in \Var N_{r;n_1,n_2,n_3}$ with the following properties: For all sufficiently small $\epsilon > 0$, there exists  an open neighborhood $\Var{O}_\epsilon$ of $\tensor{O}$, such that for all tensors $\tensor{A} \in \Var{O}_\epsilon$ we have \begin{enumerate}
\item $\tensor{A}\in \Var{N}^*$ {is a valid input for a PBA}, and
\item $\omega(\tensor{A}) \ge k \epsilon^{-1}$.
\end{enumerate}
Herein, $\Var{N}^*$ is as in \cref{def_pencil_algorithm}.
\end{thm}

\subsection{The key ingredients}
The key observation is that for computing the tensor decomposition map $\tau_{r;n_1,n_2,n_3}$ every PBA computes $\widehat{\theta}$ in S2.
We will show that the condition number of $\widehat{\theta}$ is comparable to the condition number of $\tau_{r;n_1,n_2,2}$. Combining this result with the observations from \cref{sec_distribution} and \cite{BV2018b}, which both demonstrated that the condition number of the tensor decomposition map $\tau_{r;n_1,n_2,2}$ for $n_1 \times n_2 \times 2$ tensors can be much worse than the one of $\tau_{r,n_1,n_2,n_3}$ for $n_1 \times n_2 \times n_3$ tensors, motivated our proof of \cref{thm_unstable,thm_neighborhood}.

Let us consider the relation between the tensor decomposition map for $n_1\times n_2\times 2$-tensors and $\widehat{\theta}$. For brevity, we denote the manifold of $r$-nice tensors in $\R^{n_1 \times n_2 \times 2}$ by 
\[
\Var{N} := \Var{N}_{r;n_1,n_2,2}.
\]
The main intuition underpinning the proof of \cref{thm_neighborhood} is the following diagram:
\begin{equation}\label{eqn_diagram}
\begin{tikzpicture}
 \draw (-3,1)  node (b) {$\Var{N}$};
 \draw (-3,-0.5) node (c) {$\Var{N} \times \widehat{S}_{r;n_1}$};
 \draw (3,-0.5) node (d) {$\widehat{\Var{M}}_{r;n_1,n_2,2}$};
 \draw[-stealth] (b) edge node[left] {$\operatorname{Id}_\Var{N} \times \widehat{\theta}$} (c);
 \draw[-stealth] (b) edge node[above right] {$\tau_{r;n_1,n_2,2}$} (d);
 \draw[-stealth] (c) edge node[below] {$\widehat{\eta}$} (d);
\end{tikzpicture}
\end{equation}
Herein, $\widehat{\eta}$ is any map so that $\tau_{r;n_1,n_2,2} = \widehat{\eta} \circ (\operatorname{Id}_{\Var{N}}\times\widehat{\theta})$. For example, we could take the map $\widehat{\eta} = \tau_{r;n_1,n_2,2} \circ \pi_1$, where $\pi_1(x,y) = x$ projects onto the first factor.
For clearly conveying the main idea, let us imagine for a moment that $\operatorname{Id}_\Var{N} \times \widehat{\theta}$, $\widehat{\eta}$, and $\tau_{r;n_1,n_2,2}$ were smooth $(C^\infty)$ multivariate functions between Euclidean spaces. For any such functions $f,g$, we have that $\kappa[f](x) = \| J_f(x) \|_2$, where $J_f(x)$ is the Jacobian matrix of $f$ at $x$; see, e.g., \cite[Proposition 14.1]{BC2013}. Consequently, for the composite function $g \circ f$, we get
\begin{equation}\label{eqn_condition_inequality}
\kappa[g \circ f](x) = \| J_g(f(x)) J_f(x) \|_2 \le \| J_g(f(x))\|_2 \| J_f(x) \|_2 = \kappa[g](f(x)) \cdot  \kappa[f](x).
\end{equation}
It thus seems feasible to obtain lower bounds on the condition number of $f = \operatorname{Id}_{\Var{N}} \times \widehat{\theta}$ in function of the condition numbers of $g \circ f = \tau_{r;n_1,n_2,2}$ and $g = \widehat{\eta}$. The key insight is that $\widehat{\eta}$ should be chosen in such a way that it has a condition number bounded by a constant, so that $\kappa[\operatorname{Id}_{\Var{N}} \times \widehat{\theta}](\tensor{B})$ would be comparable in magnitude to $\kappa[\tau_{r;n_1,n_2,2}](\tensor{B})$.

Using the above ideas, we will rigorously prove the next lemma in the appendix, which states that the condition number of $\widehat{\theta}$ can be bounded from below by the condition number of the tensor decomposition map $\tau_{r;n_1,n_2,2}$ in some cases.

\begin{lemma} \label{lem_samecond}
Let $\nu > 0$ be sufficiently small.
Let $\tensor{B} = \sum_{i=1}^r \vect{a}_i \otimes \vect{b}_i \otimes \vect{z}_i$  be an element of $\Var{N}$. Assume that $\|\vect{a}_i\| = 1$ and $\Vert \vect{b}_i\otimes \vect{z}_i\Vert <1+\nu$ for $i=1,\ldots,r$. Let $A=[\vect{a}_i]_i$. If there exists a matrix $A'\in \R^{n_1 \times r}$ with orthonormal columns such that $\| A  - A' \|_F \le \nu$, then
\[
{\kappa[\widehat{\theta}|_{\Var{N}}](\tensor{B})
\ge \frac{\kappa[\tau_{r;n_1,n_2,2}](\tensor{B})}{10 r} - 1}.
\]
\end{lemma}
This shows that in some circumstances, the condition number of $\widehat{\theta}|_{\Var{N}}$ is proportional to the condition number of $\tau_{r;n_1,n_2,2}$ in $\R^{n_1\times n_2\times 2}$.
Unfortunately, the errors in the computation of $\widehat{\theta}|_{\Var{N}}$ cannot always be corrected, as we prove the following result in the appendix.
\begin{lemma} \label{lem_nodecrease}
Let $\nu>0$ be sufficiently small.
Let $\tensor{A} = \sum_{i=1}^r \vect{a}_i \otimes \vect{b}_i \otimes \vect{c}_i \in \Var{N}^*$ with $\|\vect{a}_i\|=1$ and factor matrices $A,B,C$. Let $\widetilde{A}\in\R^{n_1\times r}$ be a fixed matrix with unit-norm columns and let $\delta := \min_{\pi\in\mathfrak{S}_r} \| A - \widetilde{A} P_{\pi} \|_F$. If $\|\vect{b}_i\otimes \vect{c}_i\| \ge 1-\nu$ for $i=1,\ldots,r$, $\delta<1$, and there exists a matrix $A'\in\R^{n_1 \times r}$ with orthonormal columns such that $\|A - A'\|_F \le \nu$, then for every $\widetilde{B} \in \R^{n_2 \times r}$ and every $\widetilde{C} \in \R^{n_3 \times r}$ we have
\[
 \min_{\pi\in\mathfrak{S}_r} \| A \odot B \odot C  - (\widetilde{A} \odot \widetilde{B} \odot \widetilde{C}) P_\pi \|_F \ge \sqrt{\tfrac{3}{4}} (1-\nu) \, \delta.
\]
\end{lemma}

This result implies that, even if steps S3 and S4 of a PBA could perfectly recover the rank-$1$ terms, the PBA would not be able to compensate the error introduced in the computation of $\widehat{\theta}$ in step S2. Moreover, under the assumptions of \cref{lem_samecond}, the condition number of $\widehat{\theta}$ is proportional to the condition number of $\tau_{r;n_1,n_2,2}$. This indicates that the magnification of an input perturbation of a PBA will be roughly proportional to the condition number of the TDP for $n_1\times n_2\times 2$ tensors. However, we recall from \cref{sec_distribution} and \cite{BV2018c} that there is a great discrepancy between the distribution of the condition numbers of the TDPs for $n_1 \times n_2 \times n_3$ and $n_1 \times n_2 \times 2$ tensors, the latter being much larger than the former on average. This will then imply that the excess factor $\omega$ in \cref{def_mu} is large. In the next subsections, we exploit \cref{lem_samecond,lem_nodecrease} for showing that $\omega$ is actually unbounded.

\subsection{Constructing a bad tensor}
The role of $\tensor{O}$ in \cref{thm_neighborhood} will be played by the following tensor.
Let $A'=[\vect{a}_i']_{i=1}^r \in\R^{n_1 \times r}$ and $B'=[\vect{b}_i']_{i=1}^r \in \R^{n_2 \times r}$ be matrices with orthonormal columns.
Let $U$ be the $n_3 \times n_3$ matrix $U = \left[\begin{smallmatrix}Q^\perp & Q\end{smallmatrix}\right]$, where $Q^\perp$ is an $n_3\times (n_3-2)$ matrix whose columns form an orthonormal basis of the complement of the columns of $Q$.
Define the matrix with~$r$ columns
\begin{equation*}
 C' :=  U \left( I_{n_3\times r} - \frac{2}{n_3} \vect{1}_{n_3} \vect{1}_r^T \right) \diag(1,-1,\ldots,-1)
 =
\frac{2}{n_3} U
\scalebox{.8}{$\begin{bmatrix}
\frac{n_3}{2} - 1 & 1 & 1\\
-1 & 1-\frac{n_3}{2} & 1 \\
-1 & 1 & 1-\frac{n_3}{2}  \\
-1 & 1 & 1 &\cdots\phantom{0} \\
\vdots & \vdots & \vdots  \\
-1 & 1 & 1
\end{bmatrix}$},
\end{equation*}
where $\vect{1}_k \in \R^{k}$ is the vector of ones, and $I_{m\times n} = [\vect{e}_i]_{i=1}^n$, where $\vect{e}_i$ is the $i$th standard basis vector of $\R^m$. By construction, $C'$ has orthonormal columns.
The \textit{orthogonally decomposable (odeco)} tensor associated with these factor matrices is
\begin{align} \label{eqn_cot}
 \tensor{O} := \sum_{i=1}^r \vect{a}_i' \otimes \vect{b}_i' \otimes \vect{c}_i'.
\end{align}
It will satisfy the requirements in \cref{thm_neighborhood} and complete the proof of instability of PBAs.

It is a very bad omen that $\tensor{O}$ is not a valid input for PBAs. This is because the projected tensor $\rho_Q(\tensor{O})$ has a positive-dimensional family of decompositions, implying $\kappa[\tau_{r;n_1,n_2,2}] =\infty$. Indeed, we have $Q^T \vect{c}_1'= \tfrac{2}{n_3}[-1\;-1]^T$ and, since $n_3>r+1$, for all $2\leq i\leq r$ we have $Q^T\vect{c}_i'= \tfrac{2}{n_3}[1\; 1]^T$, so that the projected tensor is
$$
\rho_Q(\tensor{O}) = - \frac{2}{n_3} \vect{a}_1' \otimes \vect{b}_1' \otimes \begin{bmatrix} 1\\1\end{bmatrix} + \frac{2}{n_3}\sum_{i=2}^r \vect{a}_i' \otimes \vect{b}_i' \otimes \begin{bmatrix} 1\\1\end{bmatrix}.
$$
By \cref{lem_kappa_lower_nn2}, \cite[Corollary 1.2]{BV2017}, or \cite[Lemma 6.5]{COV2017} the condition number of $\rho(Q)$ is infinite.
By taking a neighborhood of $\tensor{O}$ the proof of \cref{thm_neighborhood} will be completed.

Let $( \tensor{O}_1, \ldots, \tensor{O}_r ) \in \Var{S}^{\times r}$ be an ordered CPD of $\tensor{O}$. Then, the next lemma states that most of the tensors that have a decomposition in
\[
 \Var{U}_\epsilon = \{ (\tensor{A}_1,\ldots,\tensor{A}_r) \in \Var{M}_{r;n_1,n_2,n_3} \subset \Var{S}^{\times r} \mid \|\tensor{A}_i - \tensor{O}_i\|_F < \epsilon,\; i=1,\ldots,r \}
\]
are valid inputs for a PBA, where $\Var{M}_{r;n_1,n_2,n_3}$ is as in \cref{prop_nice}. The lemma is proved in the appendix.
\begin{lemma} \label{lem_opendense}
 $\Var{O}_\epsilon = \Phi_r( \Var{U}_\epsilon ) \cap \Var{N}^*$ is an open subset of $\sigma_r$ with $\tensor{O} \in \overline{ \Var{O}_\epsilon }$.
\end{lemma}

The next result allows us to apply \cref{lem_samecond,lem_nodecrease} for tensors in $\Var{O}_\epsilon$.
\begin{lemma} \label{lem_localisometry}
Let $\epsilon > 0$ be sufficiently small, and let $A',B',C'$ be as in the definition of $\tensor{O}$ in \cref{eqn_cot}. Then, there exists a constant $S > 0$ so that for all $(\tensor{A}_1,\ldots,\tensor{A}_r) \in \Var{O}_\epsilon$ with factor matrices $A,B,C$, {where both $A$ and $B$ have unit-norm columns, the following bounds} holds:
 \[
  \| A - A'  \|_F \le S \epsilon,\quad \| B - B'  \|_F \le S \epsilon \text{ and } \| C - C' \|_F \le S \epsilon.
 \]
Moreover, the columns of $B\odot C$ satisfy $1- S\epsilon \leq \Vert \vect{b}_i \otimes \vect{c}_i \Vert_F \leq 1 + S\epsilon$.
\end{lemma}
This lemma is proved in the appendix.
Combining these two lemmata with \cref{lem_samecond,lem_nodecrease}, we get the following important corollary.
\begin{cor}[An $r$-nice bad tensor] \label{hyp_bad_tensor}
Let $\epsilon > 0$ be sufficiently small. Let
\(
\tensor{A} = \sum_{i=1}^r \vect{a}_i \otimes \vect{b}_i \otimes \vect{c}_i
\)
be an element of $\Var{O}_\epsilon$ such that the factor matrices $A \in \R^{n_1 \times r}$ and $B \in \R^{n_2 \times r}$ have unit-norm columns. Then,
\begin{enumerate}
 \item $\tensor{A} \in \Var{N}^*$, i.e., $\tensor{A}$ is $r$-nice and its projection $\tensor{B} = \rho_Q(\tensor{A})$ is also $r$-nice;
 \item there exists an $A' \in \R^{n_1 \times r}$ with orthonormal columns, such that $\|A  - A'\|_F \le S\epsilon$; and
 \item $B\odot C \in \R^{n_3 \times r}$ has columns whose norms are bounded by $1 - S\epsilon \le \|\vect{b}_i\otimes \vect{c}_i\| \le 1 + S\epsilon$.
\end{enumerate}
\end{cor}

\subsection{Proof of \cref{thm_neighborhood}}
Let {$\tensor{A} \in \mathcal{O}_\epsilon$ be as in \cref{hyp_bad_tensor}. Its floating-point representation is $\widetilde{\tensor{A}} := \operatorname{fl}(\tensor{A})$.} We show that the excess factor $\omega(\tensor{A})$ from \cref{def_mu} is proportional to $\epsilon^{-1}$.

We assume that the output of step S1 is the best possible numerical result when providing $\tensor{A}$ as input, namely the floating-point representation of
\(
\tensor{B} = \rho_Q(\tensor{A}) = \sum_{i=1}^r \vect{a}_i \otimes \vect{b}_i \otimes (Q^T \vect{c}_i),
\)
i.e., $\widetilde{\tensor{B}} = \operatorname{fl}(\tensor{B}) = \operatorname{fl}( \rho_Q( \tensor{A} ) )$. For streamlining the analysis, we ignore further compounding of roundoff errors, assuming the best possible case in which the PBA manages to execute steps S2, S3 and S4 exactly (perhaps by invoking an oracle).
Let $\{\vect{a}_1,\ldots,\vect{a}_r\} = \widehat{\theta}(\widetilde{\tensor{B}})$ and $A := [\vect{a}_i]_i$. Then, by the same construction as in \cref{sec_implications_cn}, we have
\[
 \min_{\pi\in\mathfrak{S}_r} \| A  - \widetilde{A} P_\pi  \|_F \lesssim \kappa[\widehat{\theta}](\tensor{B}) \cdot \| \tensor{B} - \widetilde{\tensor{B}} \|_F.
\]
In fact, a small component in the direction of the worst perturbation is expected: From the concentration-of-measure phenomenon, assuming that the perturbation $\tensor{B}-\widetilde{\tensor{B}}$ is random with no preferred direction, it follows with high probability that the component of the perturbation in the worst direction is of size comparable to $( r \dim \Var{S}_{n_1,n_2,2} )^{-\frac{1}{2}} = ( r (n_1 + n_2) )^{-\frac{1}{2}}$. See Armentano's work \cite{Armentano} for an analysis of the impact of this consideration in the observed value of the so-called stochastic condition number.
It follows that there exists a number $1 \ge \beta_1 > 0$ such that
\[
 {\min_{\pi\in\mathfrak{S}_r} \| A  - \widetilde{A} P_\pi  \|_F}
 = \beta_1 \cdot \kappa[\widehat{\theta}](\tensor{B}) \cdot \| \tensor{B} - \widetilde{\tensor{B}} \|_F \ge \beta_1 \cdot \kappa[\widehat{\theta}|_{\Var{N}}](\tensor{B}) \cdot \| \tensor{B} - \widetilde{\tensor{B}} \|_F,
\]
where the last inequality is by definition of condition numbers and restrictions of maps.
Applying \cref{lem_samecond} and using the properties from \cref{hyp_bad_tensor}, yields
\[
  {\min_{\pi\in\mathfrak{S}_r} \| A  - \widetilde{A} P_\pi  \|_F}
  \ge \frac{\beta_1}{10 r} ( \kappa[\tau_{r;n_1,n_2,2}](\tensor{B}) - 10 r ) \cdot \| \tensor{B} - \widetilde{\tensor{B}} \|_F,
\]
where $Z := Q^T C$.
Assume that the left-hand side is bounded from above by $1$.
Regardless of the particular $\{\widetilde{\vect{b}}_i \otimes \widetilde{\vect{c}}_i\}_i$ that the PBA computes in step S3, invoking \cref{lem_nodecrease} shows that after completion of step S4 the forward error satisfies
\[
\min_{\pi \in \mathfrak{S}_r}\| A \odot B \odot C - (\widetilde{A} \odot \widetilde{B} \odot \widetilde{C})P_\pi \|_F
\ge \frac{(1-S\epsilon) \beta_1 \sqrt{3/4} }{10 r} \bigl( \kappa[\tau_{r;n_1,n_2,2}](\tensor{B}) - 10 r \bigr) \cdot \| \tensor{B} - \widetilde{\tensor{B}} \|_F.
\]
Dividing both sides of this expression by $\kappa[\tau_{r;n_1,n_2,n_3}](\tensor{A}) \cdot \|\tensor{A} - \widetilde{\tensor{A}}\|_F$ gives the excess factor $\omega(\tensor{A})$:
\begin{align} \label{eqn_proof_nearly}
\omega(\tensor{A})
\ge \frac{\kappa[\tau_{r;n_1,n_2,2}](\tensor{B}) - 10 r}{\kappa[\tau_{r;n_1,n_2,n_3}](\tensor{A})} \cdot \frac{(1-S\epsilon) \beta_1 \sqrt{3} }{ 20 r } \cdot \frac{\| \tensor{B} - \widetilde{\tensor{B}} \|_F}{\|\tensor{A} - \widetilde{\tensor{A}}\|_F}.
\end{align}
We continue by bounding the factor $\| \tensor{B} - \widetilde{\tensor{B}} \|_F \|\tensor{A} - \widetilde{\tensor{A}}\|_F^{-1}$ in this inequality.
Since $\widetilde{\tensor{B}} = \operatorname{fl}(\tensor{B})$ and $\widetilde{\tensor{A}} = \operatorname{fl}(\tensor{A})$, we have in the standard model of floating-point arithmetic,
\[
 \| \tensor{A} - \widetilde{\tensor{A}} \|_F^2 = \sum_{i_1=1}^{n_1} \sum_{i_2=1}^{n_2} \sum_{i_3=1}^{n_3} ( a_{i_1,i_2,i_3} - (1 + \delta_{i_1,i_2,i_3}) a_{i_1,i_2,i_3} )^2 \le \epsilon_u^2 \| \tensor{A} \|_F^2,
\]
where $|\delta_{i_1,i_2,i_3}|\le \epsilon_u$, and
\[
 \| \tensor{B} - \widetilde{\tensor{B}} \|_F^2
 =
 \sum_{i_1=1}^{n_1} \sum_{i_2=1}^{n_2} \sum_{i_3=1}^{n_3} ( b_{i_1,i_2,i_3} - (1 + \dot{\delta}_{i_1,i_2,i_3}) b_{i_1,i_2,i_3} )^2 = \sum_{i_1=1}^{n_1} \sum_{i_2=1}^{n_2} \sum_{i_3=1}^{n_3} \dot{\delta}_{i_1,i_2,i_3}^2 b_{i_1,i_2,i_3}^2,
\]
where $|\dot{\delta}_{i_1,i_2,i_3}| \le \epsilon_u$. There exists a $\beta_2 \ge 0$ so that
\(
 \| \tensor{B} - \widetilde{\tensor{B}} \|_F = \beta_2 \epsilon_u \| \tensor{B} \|_{F}.
\)
While a detailed analysis of the value of $\beta_2$ is outside of the scope of this work, it is reasonable to assume that~$\beta_2$ is not too small,\footnote{We can take guidance from \cite{Brent1973} where the root mean squared representation error is computed for some number systems, assuming a logarithmic distribution of the real numbers $b_{i_1,i_2,i_3}$. In this case, \cite[Section~V]{Brent1973} shows that, after plugging in the parameters of double-precision IEEE floating-point arithmetic \cite[Section~2.3]{higham}, one has
\[
\frac{1}{n_1 n_2 n_3} \sum_{i_1=1}^{n_1} \sum_{i_2=1}^{n_2} \sum_{i_3=1}^{n_3} \dot{\delta}_{i_1,i_2,i_3}^2 \approx \left( \frac{308}{\sqrt{3}} \log 2 \right)^2 \epsilon_u^2.
\]
If all $b_{i_1,i_2,i_3}$ are roughly proportional, i.e., $b_{i_1,i_2,i_3}^2 \approx \frac{1}{n_1 n_2 n_3} \| \tensor{B} \|_F^2$, then $\beta_2 \approx \frac{308}{\sqrt{3}} \log 2 \approx 123$.}
say $\beta_2 \ge 10^{-1}$.
Hence, we need bounds on the norms of $\tensor{A}$ and $\tensor{B}$. To this end, the following well-known result is useful.
\[
\Vert \tensor{O}\Vert_F
= \Bigl\| (A' \otimes B' \otimes C') \sum_{i=1}^r \vect{e}_i \otimes \vect{e}_i \otimes \vect{e}_i \Bigr\|_F
= \Bigl\| \sum_{i=1}^r \vect{e}_i \otimes \vect{e}_i \otimes \vect{e}_i \Bigr\|_F = \sqrt{r},
\]
where $\vect{e}_i$ is the $i$th standard basis vector of $\R^r$.
Let $\tensor{E}_i: = \tensor{A}_i-\tensor{O}_i$. Note that $\Vert \tensor{E}_i\Vert_F\leq \epsilon$.
The norms of $\tensor{A}$ and $\tensor{B}$ are then estimated as follows:
\begin{align*}
 \|\tensor{A}\|_F
= \left\| \sum_{i=1}^r \tensor{A}_i \right\|_F
= \left\| \sum_{i=1}^r (\tensor{O}_i + \tensor{E}_i) \right\|_F
 \le \|\tensor{O}\|_F + \sum_{i=1}^r \| \tensor{E}_i \|_F = \sqrt{r}(1 + \sqrt{r}. \epsilon),
\end{align*}
Exploiting the linearity of the multilinear multiplication $\rho_Q$ we also have
\begin{align*}
 \|\tensor{B}\|_F
 = \left\| \rho_Q(\tensor{A}) \right\|_F
 = \left\|  \rho_Q(\tensor{O}) + \sum_{i=1}^r \rho_Q( \tensor{E}_i)\right\|_F
 &\ge \| \rho_Q(\tensor{O}) \|_F - \sum_{i=1}^r \| \rho_Q(\tensor{E}_i) \|_F
 \ge \| Q^T C' \|_F - r \epsilon,
\end{align*}
where we used that the $3$-flattening of $\rho_Q(\tensor{O})$ is $(Q^T C') (A' \odot B')^T$ and, since $A'$ and $B'$ have orthonormal columns, that $\|\rho_Q(\tensor{O})\|_F = \| Q^T C'\|_F$.
By construction, $Q^T C' = \tfrac{2}{n_3} \left[ \begin{smallmatrix} -1 & 1 & \dots & 1\\ -1 & 1 & \dots & 1\end{smallmatrix}\right]$, so that we have $\Vert Q^T C'\Vert_F = \sqrt{2r} \tfrac{2}{n_3}$. We have thus shown that
\begin{equation}\label{eqn_mu}
\omega(\tensor{A}) \ge \frac{\kappa[\tau_{r;n_1,n_2,2}](\tensor{B}) - 10 r}{\kappa[\tau_{r;n_1,n_2,n_3}](\tensor{A})} \cdot \frac{(1-S\epsilon) \beta_1 \beta_2 \sqrt{3} }{ 10 r } \cdot \frac{2\sqrt{2r}  }{n_3\sqrt{r}(1 + \sqrt{r} \epsilon)}.
\end{equation}
The condition number $\kappa[\tau_{r;n_1,n_2,n_3}](\tensor{A})$ is bounded as follows.
Let $\tuple{o} = ( \tensor{O}_1,\ldots,\tensor{O}_r ) \in \Var{S}^{\times r}$ be an ordered CPD of $\tensor{O}$. By \cref{lem_kappa_identity} (2), for all $\tensor{A}=\tensor{A}_1+\cdots+\tensor{A}_r$ in $\Phi_r(\Var{U}_\epsilon)$, we have
$\kappa[\tau_{r;n_1,n_2,n_3}](\tensor{A}) = \kappa(\tuple{a}),$ where $\tuple{a}=(\tensor{A}_1,\ldots,\tensor{A}_r)$ and $\kappa(\tuple{a})$ being the condition number of the tensor rank decomposition from \cref{def_kappa}. Furthermore, $\kappa(\tuple{o})=1$ by \cite[Proposition 5.2]{BV2017}. From this \cite[Theorem 1.1]{BV2017} implies that $\kappa(\tuple{o})$ is the classic spectral $2$-norm of the derivative of $\Phi_{\tuple{o}}^{-1}$
Since $\Phi_{\tuple{o}}^{-1}$ is smooth and the spectral norm is a Lipschitz-continuous function with Lipschitz constant $1$, it follows that there exists a Lipschitz constant $\ell > 0$ such that for sufficiently small $\epsilon>0$, we have
$|\kappa[\tau_{r;n_1,n_2,n_3}](\tensor{A})-\kappa[\tau_{r;n_1,n_2,n_3}](\tensor{O})| \le \ell \|\tensor{A}-\tensor{O}\|_F$ for all $\tensor{A}\in\mathcal{O}_\epsilon$. Hence,
\begin{align}\label{eqn_proof_bound_piece2}
\kappa[\tau_{r;n_1,n_2,n_3}](\tensor{A}) \le 1 + \ell \epsilon.
\end{align}
Finally, we bound $\kappa[\tau_{r;n_1,n_2,2}](\tensor{B})$.
Let $d_s = r \dim \Var{S}_{n_1,n_2,2}$, and recall that $\vect{z}_i := Q^T \vect{c}_i$.
Recall $T$ from the proof of \cref{lem_kappa_lower_nn2}, applying it to $\tensor{B}$'s CPD. Consider the next submatrix of $T$,
\[
T' := \begin{bmatrix}
I_{n_1} \otimes \vect{b}_1 \otimes \frac{\vect{z}_1}{\|\vect{z}_1\|} & \vect{a}_2 \otimes I_{n_2} \otimes \frac{\vect{z}_2}{\|\vect{z}_2\|}
\end{bmatrix}
\text{ and set }
\vect{v}' := \begin{bmatrix} \|\vect{z}_1\| \vect{a}_2 \\ \|\vect{z}_2\| \vect{b}_1 \end{bmatrix}.
\]
Note that $\|\vect{v}'\|^2 = \|\vect{z}_1\|^2 + \|\vect{z}_2\|^2$.
From the identification of condition numbers from \cref{lem_kappa_identity} and from the steps in the proof of \cref{lem_kappa_lower_nn2} it follows that
\begin{align} \label{eqn_proof_bound_kappa_down}
\kappa[\tau_{r;n_1,n_2,2}](\tensor{B})
= \left( \min_{\vect{v} \in \R^{d_s}} \frac{\| T \vect{v} \|}{\|\vect{v}\|} \right)^{-1}
\ge \frac{\|\vect{v}'\|}{\| T' \vect{v}' \|} = \frac{\|\vect{v}'\|}{\| \vect{a}_2 \otimes \vect{b}_1 \otimes (\vect{z}_1 + \vect{z}_2) \|} = \frac{\|\vect{v}'\|}{\| \vect{z}_1 + \vect{z}_2 \|}.
\end{align}
We already showed above that $\| Z - Q^T C' \|_F \le S \epsilon$ and that $\vect{z}_i' := Q^T \vect{c}_i' = \tfrac{2}{n_3} (-1)^i \left[ \begin{smallmatrix} 1 \\ 1 \end{smallmatrix}\right]$ for $i=1,2$. Note that $\vect{z}_1' + \vect{z}_2' = 0$. Consequently, we get the bounds
\begin{align*}
 \| \vect{z}_1 + \vect{z}_2 \| &= \| (\vect{z}_1' + \vect{z}_2') + (\vect{z}_1 - \vect{z}_1') + (\vect{z}_2 - \vect{z}_2') \| \le \|\vect{z}_1 - \vect{z}_1'\| + \|\vect{z}_2 - \vect{z}_2'\| \le \sqrt{2} S \epsilon, \text{ and }\\
 \|\vect{v}'\|^2 &= \|\vect{z}_1\|^2 + \|\vect{z}_2\|^2 \ge (\max\{0, \|\vect{z}_1'\| - \|\vect{z}_1 - \vect{z}_1'\|\})^2 + (\max\{0, \|\vect{z}_2'\| - \|\vect{z}_2 - \vect{z}_2'\|\})^2.
\end{align*}
Note that in the last inequality we can bound $\|\vect{z}_i'\| - \|\vect{z}_i - \vect{z}_i'\| \ge \tfrac{2\sqrt{2}}{n_3} - S\epsilon$ for $i=1,2$. Assuming that $\epsilon$ is sufficiently small, we obtain
\(
\|\vect{v}'\| \ge \frac{4}{n_3} - S\epsilon.
\)
Plugging all of these into \cref{eqn_proof_bound_kappa_down} yields
\begin{align} \label{eqn_proof_bound_piece3}
\kappa[\tau_{r;n_1,n_2,2}](\tensor{B}) \ge \frac{2}{n_3 S} \epsilon^{-1} - \frac{1}{2} = \mathcal{O}( \epsilon^{-1} ).
\end{align}
Plugging \cref{eqn_proof_bound_piece2,eqn_proof_bound_piece3}
into \cref{eqn_mu}, the proof of \cref{thm_neighborhood} is concluded. This ultimately completes the proof of \cref{thm_unstable}.

\begin{rem}
It is important to observe that the construction of the open set $\mathcal{O}_\epsilon$ depends on the projection operator $\rho_Q$ and, hence, on $Q \in \R^{n_3 \times 2}$. That is, we have shown that regardless of a choice of $Q$ that is independent of $\tensor{A}$, there exists an open set such where the PBA with projection $\rho_Q$ is unstable. The above construction does not automatically apply to situations where $Q$ is chosen \emph{as a function} of the input $\tensor{A}$.
\end{rem}

\section{Numerical experiments} \label{sec_numex}
We present the results of some numerical experiments in Matlab R2017b for supporting the main result and exemplifying the behavior of PBAs on third-order random CPDs. They were performed on a computer system consisting of two Intel Xeon E5-2697 CPUs (12 cores, 2.6GHz each) with 128GB of main memory.

Three PBAs are considered in the experiments below, which we refer to as \texttt{cpd\_pba}, \texttt{cpd\_pba2} and \texttt{cpd\_gevd}, respectively.\footnote{Both \texttt{cpd\_pba} and \texttt{cpd\_pba2} are provided in the ancillary files to this manuscript; they require some functionality from Tensorlab v3.0 \cite{Tensorlab}.} The first, \texttt{cpd\_pba}, is an ordinary implementation of the prototypical PBA discussed in \cref{sec_introduction}, using ST-HOSVD \cite{VVM2012} as orthogonal Tucker compression.
{\texttt{cpd\_pba2} computes the CPD by randomly projecting the input tensor $\tensor{A}$ with $\rho_Q$, then employing the \texttt{cpd} function from Tensorlab v3.0 to recover the two factor matrices $A$ and $B$, and finally computing $A \odot B \odot ((A\odot B)^\dagger \tensor{A}_{(3)}^T)^T$ to obtain a representative of the set of rank-$1$ tensors.}
The last PBA we consider is the \texttt{cpd\_gevd} function from Tensorlab v3.0. The analysis in \cref{sec_3o_unstable} does not strictly apply to the default settings\footnote{There is an option to use a random orthonormal projection, in which case the theory of this paper applies.} of \texttt{cpd\_gevd}, because it chooses the projection matrix $Q$ as a function of the input tensor $\tensor{A} \in \R^{n_1 \times n_2 \times n_3}$. Specifically, if $(U_1, U_2, U_3) \cdot \tensor{S} = \tensor{A}$ is the HOSVD \cite{Lathauwer2000}, then \texttt{cpd\_gevd} chooses $Q$ as the first two columns of $U_3$.

Throughout these experiments, the forward error of the TDP is evaluated as follows.
If $\tensor{A} = \sum_{i=1}^r \vect{a}_i \otimes \vect{b}_i \otimes \vect{c}_i$ and $\tensor{A}' = \sum_{i=1}^r \vect{a}_i' \otimes \vect{b}_i' \otimes \vect{c}_i'$, then we recall from \cref{eqn_ruleofthumb} that
\[
\mathrm{err}_\mathrm{forward} := \min_{\pi \in \mathfrak{S}_r} \| A \odot B \odot C - (A' \odot B' \odot C')P_\pi \|_F,
\]
is the forward error. Evaluating all $r!$\ permutations is a Herculean task when $r \gg 10$. Fortunately, when $\tensor{A}$ and $\tensor{A}'$ are very close, the optimal permutation can be found heuristically by solving the linear least-squares problem $\min_{X \in \R^{r \times r}} \| A\odot B\odot C - (A'\odot B'\odot C') X \|_F$ and then projecting the minimizer to the set of permutation matrices by setting the largest value in every row to $1$ and the rest to zero. In all experiments, the forward error is computed in this manner.

\subsection{The bad odeco tensor}
We start with an experiment to support the analysis of \cref{sec_3o_unstable}. Let $\rho_Q = \operatorname{Id} \otimes \operatorname{Id} \otimes Q^T$, where $Q \in \R^{n_3 \times 2}$, be the projection operator of the PBA. Let $\tensor{A}\in\R^{n_1 \times n_2 \times n_3}$ be an $r$-nice tensor whose CPD is $\epsilon$-close to the odeco tensor \cref{eqn_cot}, i.e.,
$\tensor{A} \in \Var{O}_\epsilon$, where the latter is as in \cref{lem_opendense}. According to the analysis in \cref{sec_3o_unstable}, the excess factor $\omega(\tensor{A})$ of a PBA with projection operator $\rho_Q$ should behave like $\mathcal{O}(\epsilon^{-1})$.

\begin{figure}
\centerline{\includegraphics[width=.8\textwidth]{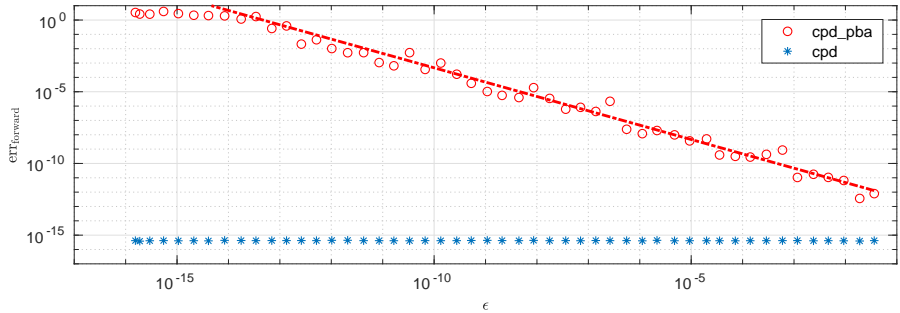}}
\caption{The forward error $\textrm{err}_\textrm{forward}$ of \texttt{cpd\_pba} and \texttt{cpd} for random tensors in $\Var{O}_\epsilon$ from \cref{lem_opendense} in function of $\epsilon$. The dashed line is $3.5 \cdot 10^{-14} \cdot \epsilon^{-1}$.}
\label{fig_cot_error}
\end{figure}

We consider $89 \times 29 \times 11$ tensors. $Q \in \R^{11 \times 2}$, $A' \in \R^{89 \times 10}$ and $B' \in \R^{29 \times 10}$ were respectively generated by computing the $Q$-factor of the $QR$-decomposition of a matrix with i.i.d.\ standard normal entries. The matrix $C' \in \R^{11 \times 10}$
was constructed as in the definition of \cref{eqn_cot}. Then, $\tensor{O}_i := \vect{a}_i' \otimes \vect{b}_i' \otimes \vect{c}_i'$ for $i=1,\ldots,10$. For $k = 1, \ldots, 50$, we constructed the randomly perturbed tensors $\tensor{P}_{k,i} = \tensor{O}_i + 2^{-k} \frac{\tensor{X}_{k,i}}{\|\tensor{X}_{k,i}\|_F}$,
where $\tensor{X}_{k,i}$ has i.i.d.\ standard normal entries. Using the \texttt{cpd} function with default settings from Tensorlab, we then computed the rank-$1$ approximations $\tensor{A}_{k,i}$ of $\tensor{P}_{k,i}$. Let $\epsilon_k := \max\,\{ \|\tensor{A}_{k,i} - \tensor{O}_i\|_F \}_i$, and then the corresponding tensor is $\tensor{A}_k = \sum_{i=1}^{10} \tensor{A}_{k,i}$,
so that $\tensor{A}_k \in \mathcal{O}_{\epsilon_k}$ with probability~$1$. Let $\tuple{a}_k^* = \{\tensor{A}_{k,1},\ldots,\tensor{A}_{k,10}\}$ denote the true CPD. A rank-$10$ CPD $\tuple{a}_k \in \Var{S}^{\times r}/\mathfrak{S}_r$ of $\tensor{A}_k$ was computed numerically using \texttt{cpd\_pba} and the forward error relative to $\tuple{a}_k^*$ was computed. We also applied \texttt{cpd} with default settings to $\tensor{A}_k$ for numerically computing another rank-$10$ CPD $\tuple{a}_k'$. The forward error between $\tuple{a}_k^*$ and $\tuple{a}_k'$ was recorded.

The results of the above experiment are shown in \cref{fig_cot_error}. \texttt{cpd} attains a forward error of approximately $4 \cdot 10^{-16}$ in all cases. As the random tensors are very close to the odeco tensor, their condition numbers are approximately $1$. A forward error equal to a small multiple of the machine precision $1.11 \cdot 10^{-16}$ is thus anticipated from a stable algorithm. The situation is dramatically different for \texttt{cpd\_pba}. Since the odeco tensor was chosen to behave badly with respect to the projection $\rho_Q$, we expect from \cref{sec_3o_unstable} that the forward error of the PBA grows like the excess factor $\omega = \mathcal{O}( \epsilon^{-1} )$.
The dashed line in \cref{fig_cot_error} shows the result of fitting the model $k \epsilon^{-1}$ to the data with $\epsilon > 10^{-14}$. As can be seen, the experimental data match the predictions from the theory in \cref{sec_3o_unstable} very well, specifically with regard to the growth rate of the excess factor.

\subsection{Distribution of the excess factors}
The previous experiment illustrated the forward error in worst possible case that we know of, mainly to illustrate \cref{thm_unstable}. Based on the construction in \cref{sec_3o_unstable}, it is not reasonable to expect that this will correspond to the typical behavior.
However, the next experiment shows that, unfortunately, one should typically expect a loss of precision of at least a few digits.

\begin{figure}[tb]
\centerline{\subfloat{\includegraphics[height=1.9in]{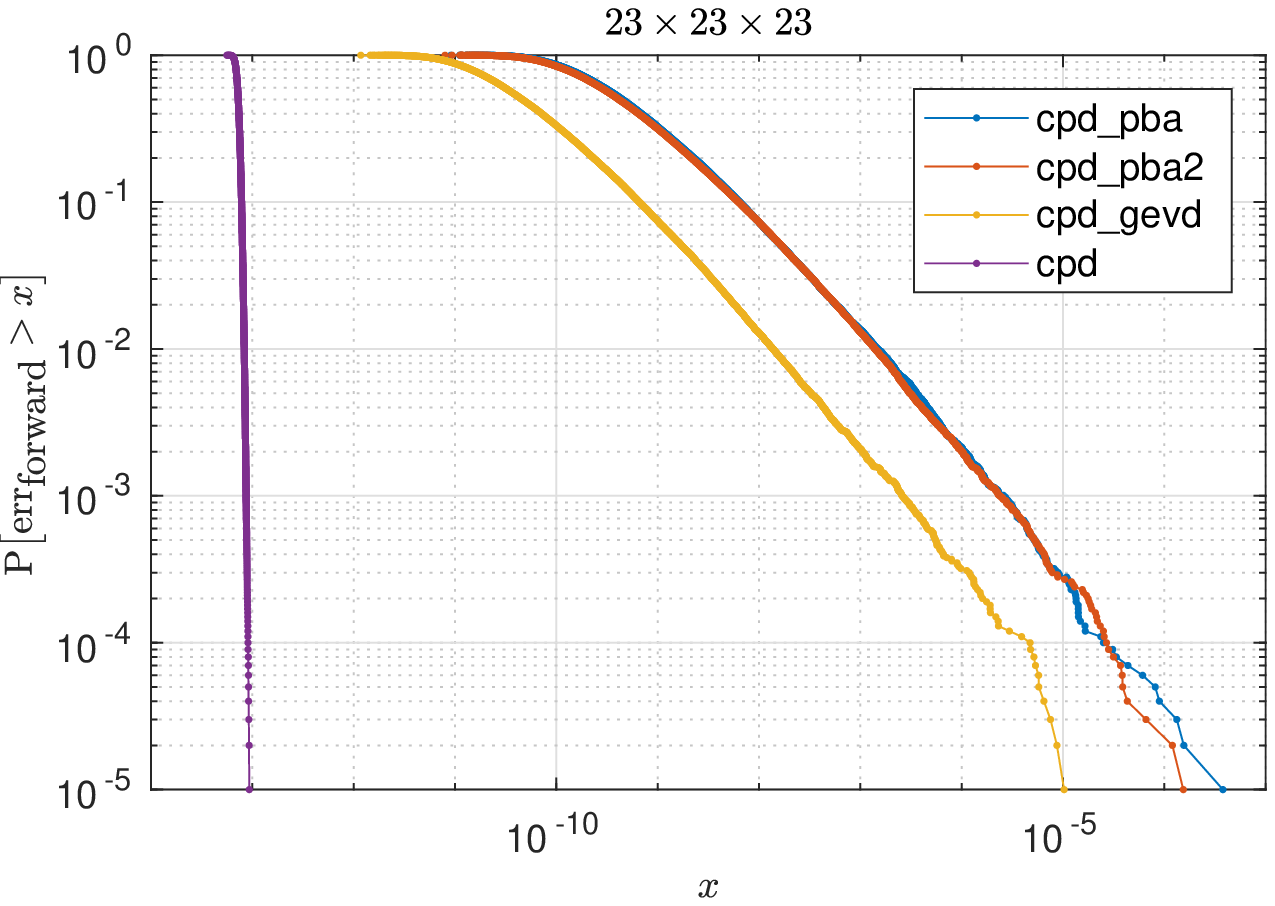}} \hspace{1em} \subfloat{\includegraphics[height=1.9in]{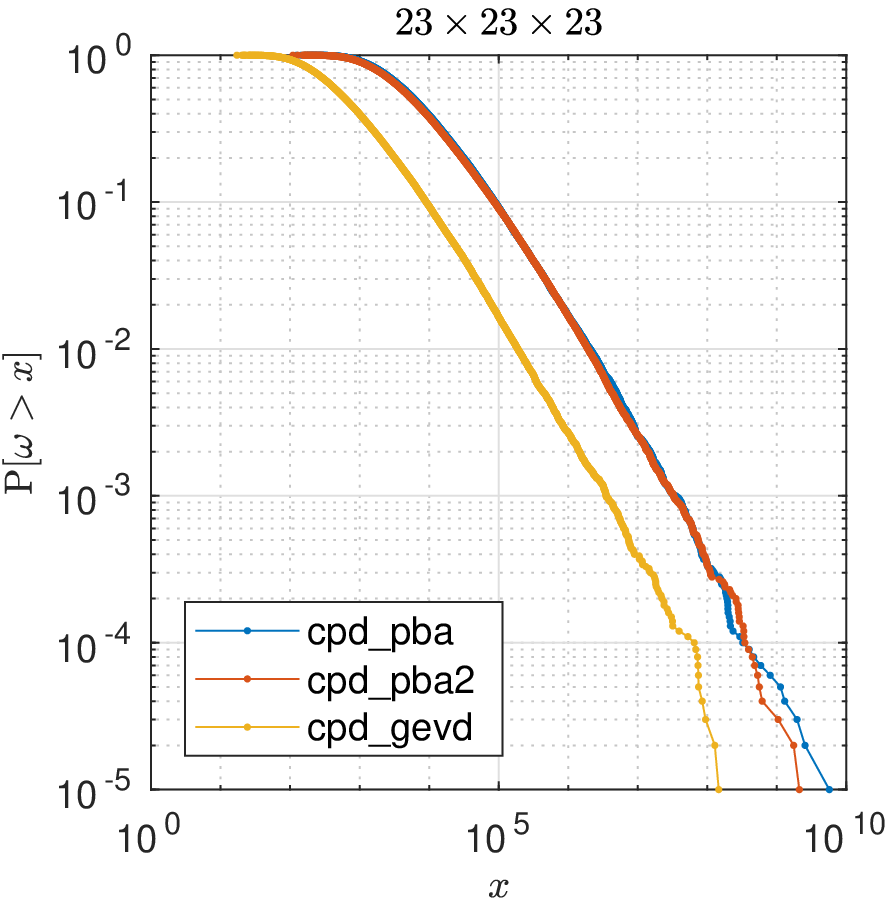}}}
\centerline{\subfloat{\includegraphics[height=1.9in]{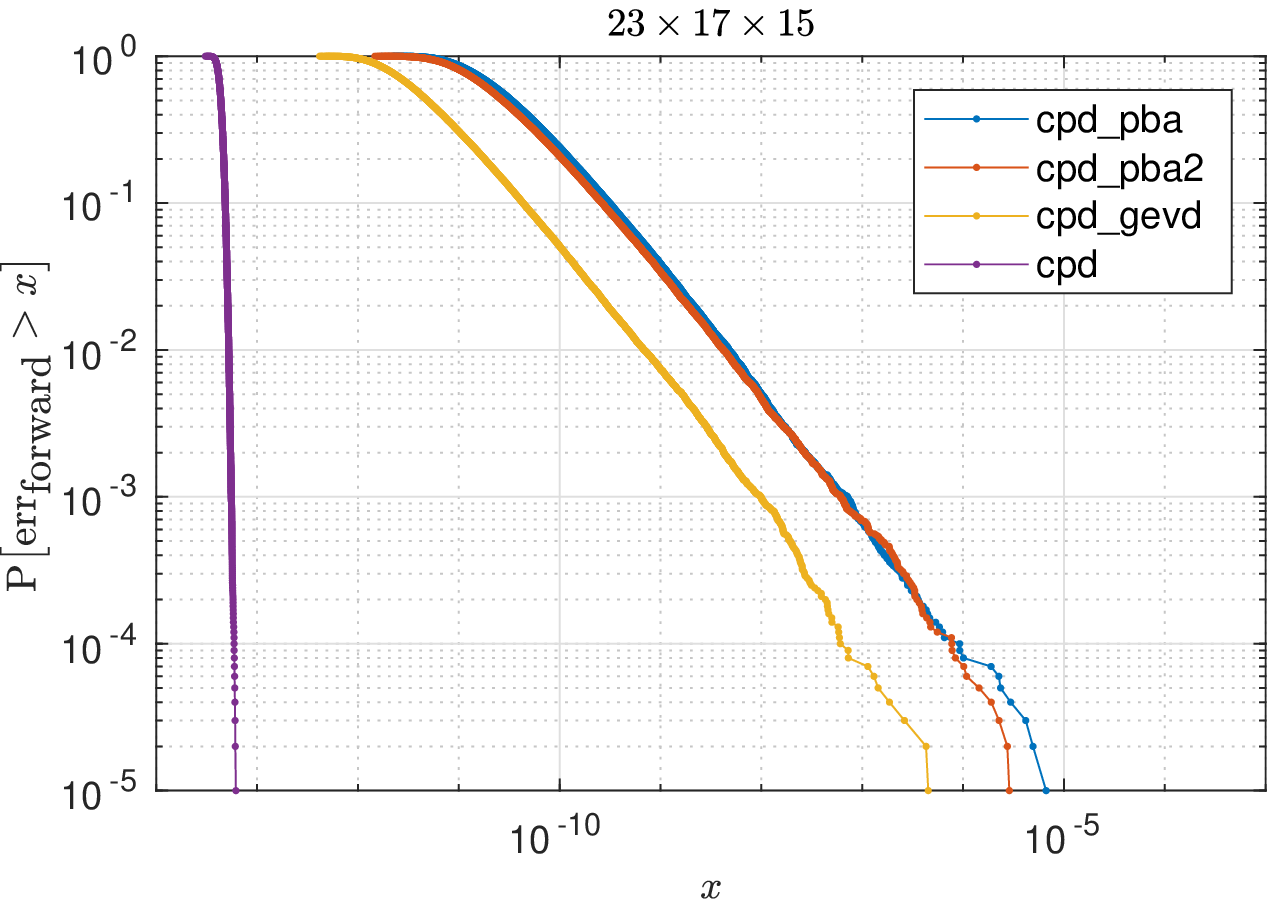}} \hspace{1em} \subfloat{\includegraphics[height=1.9in]{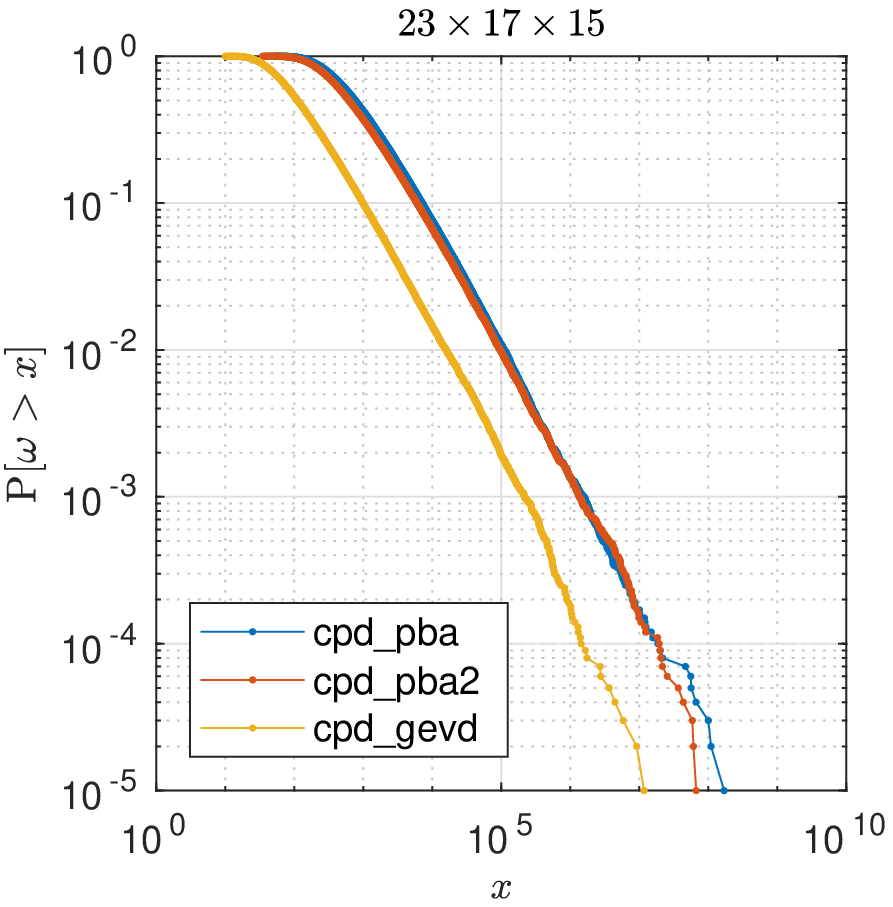}}}
\centerline{ \subfloat{\includegraphics[height=1.9in]{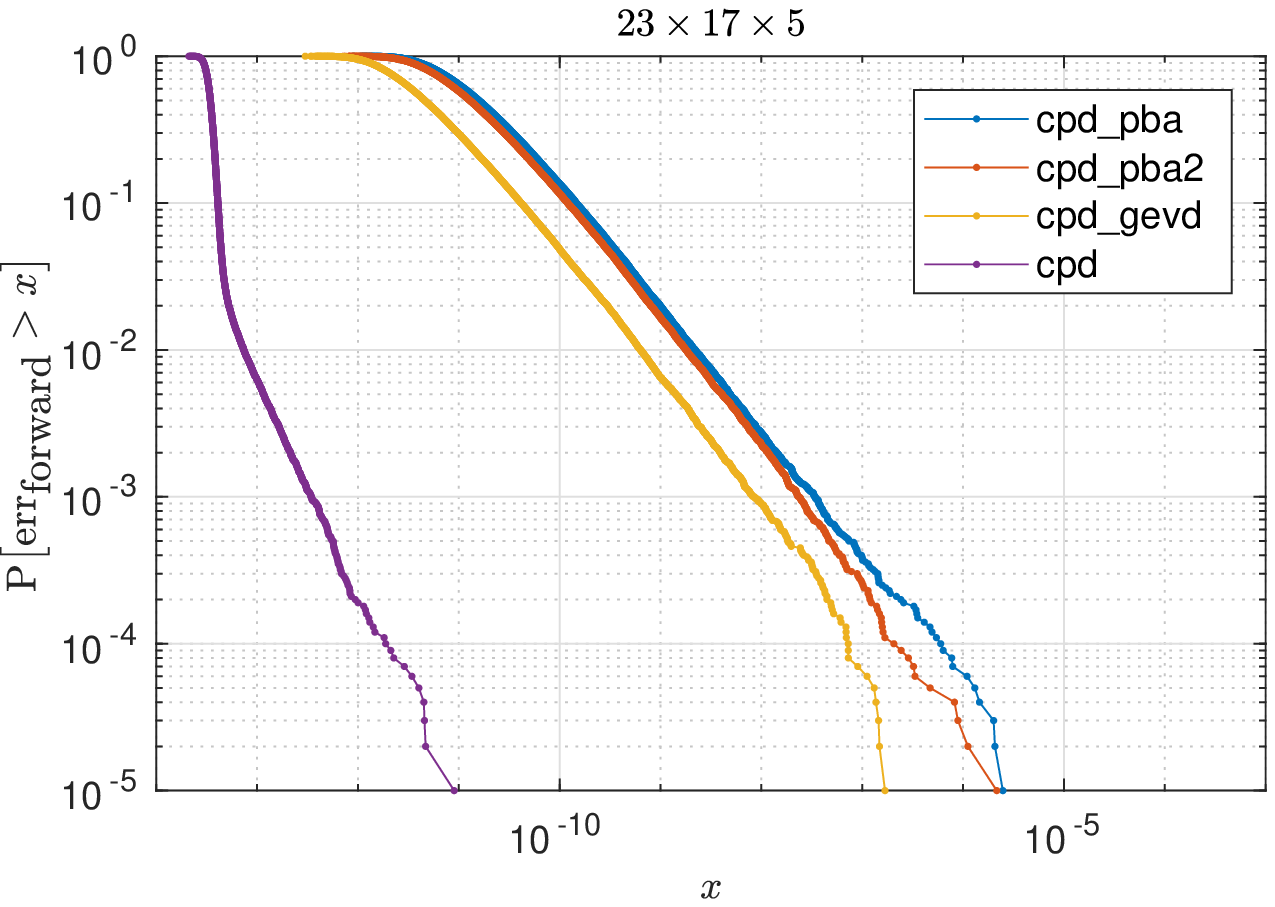}} \hspace{1em} \subfloat{\includegraphics[height=1.9in]{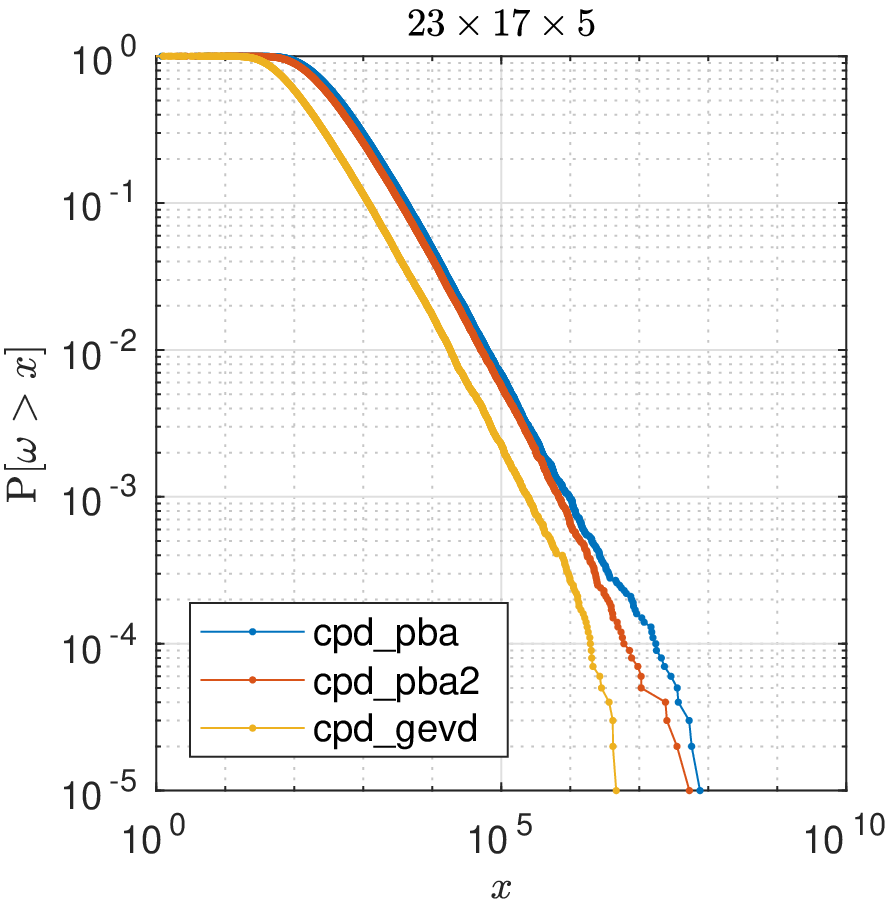}} }
\caption{The empirical cumulative distribution function of the forward error $\text{err}_\text{forward}$ and the multiplication factor $\mu$ for the standard PBA from \cref{sec_introduction}, the \texttt{cpd\_gevd} and \texttt{cpd} functions from Tensorlab, and the \texttt{cpd} function from Tensorlab initialized with the factor matrices obtained with the PBA applied to rank-$n_2$ tensors of size $n_1 \times n_2 \times n_3$.}
\label{fig_algorithm_errors}
\end{figure}

The setup is as follows. For each tested tensor shape $n_1 \times n_2 \times n_3$, we generated $10^5$ random rank-$n_2$ CPDs $\{ \vect{a}_1\otimes\vect{b}_1\otimes\vect{c}_1, \ldots, \vect{a}_r\otimes\vect{b}_r\otimes\vect{c}_r \}$
by sampling the entries of the vectors $\vect{a}_i \in \R^{n_1}$, $\vect{b}_i \in \R^{n_2}$ and $\vect{c}_i \in \R^{n_3}$ i.i.d.\ from a standard normal distribution. The corresponding tensor $\tensor{A} = \sum_{i=1}^r \vect{a}_i \otimes \vect{b}_i \otimes \vect{c}_i$ was then constructed. We used the three PBAs as well as Tensorlab's \texttt{cpd}
function to compute the CPD from $\tensor{A}$, recording the forward error. The results are displayed in \cref{fig_algorithm_errors}. The plots on the left show the empirical ccdfs of the forward errors of the four algorithms. The plots on the right show the excess factors of the PBAs.

Recall that \texttt{cpd} by default will use the PBA \texttt{cpd\_gevd} as initialization and will then refine its output by running a quasi-Newton method; see \cite{Sorber2013,Tensorlab}.
The stopping criterion for \texttt{cpd} was set to $\| \tensor{A} - \sum_{i=1}^r \tensor{A}_i' \|_F \le 2 \sqrt{10} \epsilon_u$, where $\epsilon_u \approx 1.1 \cdot 10^{-16}$ is the unit roundoff of standard double precision floating point arithmetic, and $\tensor{A}_i'$ are the rank-$1$ tensors. The forward error of \texttt{cpd} will thus be bounded approximately by $2 \sqrt{10} \kappa(\tensor{A}_1,\ldots,\tensor{A}_r) \cdot \epsilon_u$.
Recalling the shape of the ccdfs of the condition number from \cref{fig_cond_distribution}, we again note that as $n_3$ increases, the likelihood of large condition numbers diminishes. In fact, most of the generated TDPs were well-conditioned, as {can be inferred from the figure by noting that} the forward error of \texttt{cpd} is always less than $10^{-11}$.

The loss of precision of the two PBAs is very pronounced in \cref{fig_algorithm_errors}. Although \texttt{cpd\_gevd} is not strictly a PBA, because its projection operator depends on the tensor, its loss of precision in \cref{fig_algorithm_errors} asymptotically matches that of the PBAs.
Note the seemingly asymptotic log-linear relationship between the probability $\mathrm{P}[\omega > x]$ and $x$ in the right plots in \cref{fig_algorithm_errors}; that is, it seems plausible that asymptotically $\mathrm{P}[\omega > x] = a x^{-1}$ for some $a>0$.
A possible explanation of this behavior follows from our geometrical interpretation of the causes of instability. The inputs $\tensor{A}$ for which we expect $\omega(\tensor{A}) > x$ with large $x$ are those such that $\vect{c}_i\not\approx \vect{c}_j$ and yet $Q^T \vect{c}_i\approx Q^T \vect{c}_j$ for some $i\neq j$. This is more likely to happen if $n_3$ is large, since $\vect{c}_i\in\mathbb{S}(\R^{n_3})$ and $Q^T \vect{c}_i\in\mathbb{S}(\R^{2})$.
Indeed, the extreme case $Q^T \vect{c}_i = Q^T \vect{c}_j$, for some $i\neq j$, corresponds to a hypersurface $\mathcal{L}$ of $\mathbb{S}(\R^{n_3})^{\times r}$. If we {realize} that $Q^T \vect{c}_i\approx Q^T \vect{c}_j$ is similar to the property of being close to $\mathcal{L}$, then we expect $\omega >x $ to happen in some neighborhood of radius comparable to $1/x$ around $\mathcal{L}$. This neighborhood will have a volume of the order of $x^{-1}$, qualitatively explaining the observed behavior.

\section{Conclusions} \label{sec_conclusions}
We proved in \cref{thm_unstable} that popular pencil-based algorithms for computing the CPD of low-rank third-order tensors are numerically unstable. Moreover, not only do there exist inputs for which such algorithms are unstable, the numerical experiments suggest that for certain random CPDs the loss of precision is roughly $\mathcal{O}( - \log_{10}( \epsilon ) )$ with probability $\epsilon$. In addition to these results, we bounded the distribution of condition numbers of random CPDs, in \cref{thm_probability_dist_est}.

The main conclusion of our work is this: \textit{PBAs should be handled with care}, as the numerical experiments in \cref{sec_numex} demonstrated that an excess loss of precision is probable. When the most accurate result is sought, we advise to apply a Newton-based refinement procedure to the output of a PBA. This is in fact the default strategy pursued by the \texttt{cpd} function from Tensorlab v3.0.
While this strategy is certainly advisable when the input is perturbed only by roundoff errors, it is not clear to us whether employing a PBA for generating a starting point for an iterative method is more effective than a random initialization in the presence of significant (measurement) errors in the input data, both for reasons of conditioning (\cref{thm_probability_dist_est}) and stability (\cref{thm_unstable}). We believe that a further study on this point is required.

We hope that the construction of inputs for which PBAs are unstable, in \cref{sec_3o_unstable}, offers insights that can help in the design of numerically stable algorithms for computing CPDs.
Our analysis suggests that methods partly recovering the rank-$1$ tensors from a matrix pencil are numerically unstable in the neighborhood of some adversarially chosen inputs.

Finally, we emphasize that the reason why PBAs are numerically unstable is caused by transforming the tensor decomposition problem into a \textit{more difficult computational problem} that is nevertheless \textit{perceived to be easier to solve}, probably because there are direct algorithms for solving them. Here is thus a decidedly positive message that we wish to stress: computing a CPD can be easier, from a numerical point of view, than solving the generalized eigendecomposition problem for a projected tensor. We hope that these observations may (re)invigorate the search for numerically stable algorithms for computing CPDs.

\appendix
\section{Proof of the lemmata}

The proofs of the technical \cref{lem_kappa_identity,lem_opendense,lem_localisometry,lem_samecond,lem_nodecrease} are presented.

\subsection{Proof of \cref{lem_kappa_identity}}
For brevity, let
$$\Var{M} = \Var{M}_{r;n_1,\ldots,n_d},\; \widehat{\Var{M}} = \widehat{\Var{M}}_{r;n_1,\ldots,n_d},\;\Var{N} = \Var{N}_{r;n_1,\ldots,n_d}, \; \text{ and }\;\tau = \tau_{r;n_1,\ldots,n_d}.$$
For (1) we just refer to \cite[Section 2.3]{Petersen} which covers our case since the group $\mathfrak{S}_r$ acts by isometries on $\Var{M}$. Therefore, the induced metric $\widehat{g}$ on $\widehat{\Var{M}}$ is the pushforward $\widehat{g} := \widehat{\pi}_* g$  of the Riemannian metric $g$ on $\Var{M}$ that is inherited from the standard product of inner products on the ambient Euclidean space (namely $(\R^{n_1 \cdots n_d})^{\times r}$) of $\Var{M} \subset \Var{S}^{\times r} \subset (\R^{n_1 \cdots n_d})^{\times r}$.
We denote by $h$ the metric on $\Var{N}$ which is given by the standard Euclidean inner product $\langle \cdot, \cdot \rangle$ that $\Var{N}$ inherits from the ambient space $\R^{n_1 \times \cdots \times n_d} \simeq \R^{n_1 \cdots n_d}$.

It will be insightful to describe the metric $\widehat{g}$ on $\widehat{\Var{M}}$ more concretely. Let $\tuple{a} = (\tensor{A}_1,\ldots,\tensor{A}_r) \in \Var{M}$ be an arbitrary ordered $r$-nice decomposition, and let $\widehat{\tuple{a}} := \widehat{\pi}(\tuple{a})$ denote the corresponding CPD.
Let $\widehat{\pi}^{-1}_{\tuple{a}}$ be the smooth local section with $(\widehat{\pi}^{-1}_{\tuple{a}} \circ \widehat{\pi})(\tuple{a}) = \tuple{a}$. The pushforward $\widehat{g} = \widehat{\pi}_* g$ is defined (see \cite[p.~183]{Lee2013}) as the map satisfying
\(
\widehat{g}_{\,\widehat{\tuple{a}}} ( \widehat{\tuple{s}}, \widehat{\tuple{t}} ) := g_{\tuple{a}} ( \deriv{\widehat{\pi}_{\tuple{a}}^{-1}}{\widehat{\pi}(\tuple{a})} (\widehat{\tuple{s}}), \deriv{\widehat{\pi}_{\tuple{a}}^{-1}}{\widehat{\pi}(\tuple{a})} ( \,\widehat{\tuple{t}}\, ) )
\)
for all
\(\widehat{\tuple{s}}, \widehat{\tuple{t}} \in \Tang{\widehat{\tuple{a}}}{\widehat{\Var{M}}} \simeq \Tang{\tuple{a}}{\Var{M}} 
\)
where
$g_{\tuple{a}}( \tuple{b}, \tuple{c} ) := \sum_{i=1}^r \langle \tuple{b}_i, \tuple{c}_i \rangle$ with $\tuple{b}_i, \tuple{c}_i \in \Tang{\tensor{A}_i}{\Var{S}}$.
Using the identification $\Tang{\widehat{\tuple{a}}}{\widehat{\Var{M}}} \simeq \Tang{\tuple{a}}{\Var{M}}$ which is given by the isometry $\deriv{\widehat{\pi}}{\tuple{a}}$ we can denote $\widehat{\tuple{t}} = \{ \vect{t}_1, \ldots, \vect{t}_r \}$ with $\vect{t}_i \in \Tang{\tensor{A}_i}{\Var{S}}$. Similarly, we can write $\widehat{\tuple{s}} = \{ \vect{s}_1, \ldots, \vect{s}_r \}$ with $\vect{s}_i \in \Tang{\tensor{A}_i}{\Var{S}}$. Then, it follows that
\[
g_{\,\widehat{\tuple{a}}} ( \widehat{\tuple{s}}, \widehat{\tuple{t}} ) = \sum_{i=1}^r \langle \vect{t}_i, \vect{s}_i \rangle \text{ and }
\| \tuple{t} \|_{\widehat{\Var{M}},\widehat{\tuple{a}}}^2 = \sum_{i=1}^r \|\vect{t}_i\|^2 = \bigl\| \begin{bmatrix} \vect{t}_1 & \cdots & \vect{t}_r \end{bmatrix} \bigr\|_F^2,
\]
where $\| \tuple{t} \|_{\widehat{\Var{M}},\widehat{\tuple{a}}}$ is the induced norm on $\Tang{\widehat{\tuple{a}}}{\widehat{\Var{M}}}$.

From the foregoing discussion it indeed follows for every choice of $\tuple{a}\in\widehat{\pi}^{-1}(\tau(\tensor{A}))$ that
\begin{align*}
\kappa[\tau](\tensor{A})
&= \max_{\vect{t} \in \Tang{\tensor{A}}{\Var{N}}} \frac{\| (\deriv{\tau}{\tensor{A}})(\vect{t}) \|_{\widehat{\Var{M}},\tau(\tensor{A})}}{\| \vect{t} \|_F}
= \max_{\vect{t} \in \Tang{\tensor{A}}{\Var{N}}} \frac{\| \deriv{\widehat{\pi}_{\tuple{a}}^{-1}}{\widehat{\pi}(\tuple{a})} \bigl( (\deriv{\tau}{\tensor{A}})(\vect{t}) \bigr) \|_F }{\| \vect{t} \|_F} \\
&= \max_{\vect{t} \in \Tang{\tensor{A}}{\Var{N}}} \frac{\| (\deriv{(\widehat{\pi}_{\tuple{a}}^{-1} \circ \tau)}{\tensor{A}})(\vect{t}) \bigr) \|_F}{\| \vect{t} \|_F}
= \max_{\vect{t} \in \Tang{\tensor{A}}{\Var{N}}} \frac{\| (\deriv{\Phi_{\tuple{a}}^{-1}}{\tensor{A}})(\vect{t}) \|_F}{\| \vect{t} \|_F}
= \kappa( \tensor{A}_1, \ldots, \tensor{A}_r ),
\end{align*}
where the second equality is by the definition of the metric, the third by the linearity of derivatives,
and the final equality is precisely Theorem 1.1 of \cite{BV2017}. This finishes the proof of (2).

Finally, (3) follows from the fact that $\widehat{\pi}$ is a local isometry and thus preserves the lengths of curves. Given any curve joining two elements in $\widehat{\Var{M}}$, its lift through the covering $\widehat{\pi}$ thus has the same length. Since we are free to choose the representative, we thus choose one that minimizes the length of the lifted curve. \qed

\subsection{Proof of Lemma \ref{lem_samecond}}
For brevity, we drop all subscripts:
$$\Var{N} = \Var{N}_{r;n_1,n_2,2},\; \widehat{\Var{M}} = \widehat{\Var{M}}_{r;n_1,n_2,2},\;\Var{M} = \Var{M}_{r;n_1,n_2,2},\;\widehat{S} = \widehat{S}_{r;n_1},\;S = S_{r;n_1} \text{ and } \tau = \tau_{r;n_1,n_2,2}.$$ 
Consider again the diagram from \cref{eqn_diagram}.
Note that $\Var{N}$, $\widehat{\Var{M}}$, and $\Var{N}\times\widehat{S}$ are manifolds. We claim that $\Theta = \operatorname{Id}_{\Var{N}} \times \theta|_{\Var{N}}$ and $\widehat{\eta}$ are smooth maps between manifolds. We can explicitly write $\widehat{\eta}$ as
$$
\widehat{\eta} : \Var{N} \times \widehat{S} \to \widehat{\Var{M}},\;
(\tensor{B}, \{\vect{a}_1,\ldots,\vect{a}_r\}) \mapsto \widehat{\pi}( A\odot (A^\dagger \tensor{B}_{(1)})^T ),
$$
where $A = [\vect{a}_i]_i \in S$ is a $n_1 \times r$ matrix with the $\vect{a}_i$'s as columns in any order; $\tensor{B}_{(1)} = A (B \odot Z)^T$ is the $1$-flattening \cite{Landsberg2012} of $\tensor{B} = \sum_{i=1}^r \sten{a}{i}{} \otimes \sten{b}{i}{} \otimes \vect{z}_i$;
and with a minor abuse of notation $\widehat{\pi}$ is the smooth map that takes a matrix and sends it to the set of its columns. By assumption $r \le n_1$ so that ${S}$ is the manifold of matrices with linearly independent unit-norm columns. Therefore, $A^\dagger = (A^T A)^{-1} A^T$ for all $A \in S$, which is a smooth map. Consequently, $\widehat{\eta}$ is a smooth map, by \cite[Proposition 2.10 (d)]{Lee2013}.
Let $\Psi_{n_1,\ldots,n_d}^*$ be the map from \cref{eqn_psi_star}. Then, we have
\[
 \widehat{\theta}|_{\Var{N}} =  \pi \circ \bigl( \pi_2 \circ (\Psi_{n_1,n_2,2}^*)^{-1} \bigr)^{\times r} \circ \tau,
\]
where $\pi_2 : \R\setminus\{0\} \times \mathbb{S}^+(\R^{n_1}) \times \mathbb{S}^+(\R^{n_2}) \times \mathbb{S}^+(\R^{n_3}) \to \mathbb{S}^+(\R^{n_1})$ projects onto the second factor. The projection $\pi$ is a local diffeomorphism by \cref{lemma_quotient1}, the coordinate projection $\pi_2$ is smooth, $\Psi_{n_1,n_2,2}^*$ is a diffeomorphism, and $\tau$ is a diffeomorphism by \cref{prop_image}.
Therefore, $\theta|_{\Var{N}}$ is smooth, by \cite[Proposition 2.10(d)]{Lee2013}, and so $\Theta$ is smooth.

Recall that the spectral norm of a linear operator $F : V \to W$, where $V$ and $W$ are normed vector spaces with respective norms $\|\cdot\|_V$ and $\|\cdot\|_W$, is
\(
 \| F \|_{V,W} := \max_{\vect{t} \in V} \tfrac{\|F(\vect{t})\|_W}{\|\vect{t}\|_V}.
\)
For composable maps, the foregoing spectral norms are submultiplicative.
Since $\tau = \Theta \circ \widehat{\eta}$ is a composition of smooth maps between manifolds, we have that $\deriv{\tau}{\tensor{A}} = \deriv{\widehat{\eta}}{\Theta(\tensor{A})} \circ \deriv{\Theta}{\tensor{A}}$.
Therefore,
\begin{align*}
 \kappa[\tau](\tensor{A})
:= \| \deriv{\tau}{\tensor{A}} \|_{\Tang{\tensor{A}}{\Var{N}}, \Tang{\tau(\tensor{A})}{\widehat{\Var{M}}} }
&\le \| \deriv{\Theta}{\tensor{A}} \|_{ \Tang{\tensor{A}}{\Var{N}}, \Tang{\Theta(\tensor{A})}{(\Var{N}\times\widehat{S}})} \; \| \deriv{\widehat{\eta}}{\Theta(\tensor{A})} \|_{\Tang{\Theta(\tensor{A})}{(\Var{N}\times\widehat{S}}), \Tang{\tau(\tensor{A})}{\widehat{\Var{M}}} } \\
&= \kappa[\Theta](\tensor{A}) \cdot \kappa[\widehat{\eta}](\Theta(\tensor{A})),
\end{align*}
where the last step follows from the definition in \cref{def_kappa_general}. Note that this generalizes \cref{eqn_condition_inequality}.

We can write the condition number of $\Theta$ as a function of the condition number of $\widehat{\theta}|_{\Var{N}}$. Indeed, let $\vect{t} \in \Tang{\tensor{A}}{\Var{N}}$ be arbitrary, and observe that
\[
\| \deriv{ \Theta }{\tensor{A}}(\vect{t}) \|_{\Var{N} \times \widehat{S},\Theta(\tensor{A})}^2
 = \| \bigl(\vect{t}, \deriv{\widehat{\theta}|_{\Var{N}}}{\tensor{A}}(\vect{t}) \bigr) \|_{\Var{N} \times \widehat{S},\Theta(\tensor{A})}^2
 = \|\vect{t}\|_2^2 +  \| \deriv{\widehat{\theta}|_{\Var{N}}}{\tensor{A}}(\vect{t}) \bigr) \|_{\widehat{S},\widehat{\theta}|_{\Var{N}}(\tensor{A})}^2.
\]
As a result, we find
\[
 \bigl( \kappa[\Theta](\tensor{A}) \bigr)^2
 = \max_{\vect{t} \in \mathbb{S}(\Tang{\tensor{A}}{\Var{N}})} \| \deriv{\Theta}{\tensor{A}}(\vect{t}) \|_{\Var{N} \times \widehat{S},\Theta(\tensor{A})}^2
 =  1 + \max_{\vect{t} \in \mathbb{S}(\Tang{\tensor{A}}{\Var{N}})} \| \deriv{\widehat{\theta}|_{\Var{N}}}{\tensor{A}}(\vect{t}) \bigr) \|_{\widehat{S},\widehat{\theta}|_{\Var{N}}(\tensor{A})}^2
 = 1 + \bigl( \kappa[\widehat{\theta}|_{\Var{N}}](\tensor{A}) \bigr)^2.
\]
Exploiting that $\sqrt{1+x^2} \le 1+|x|$ for all $x\in\R$, we thus find
\begin{align}\label{eqn_proof_end}
 \frac{\kappa[\tau](\tensor{A})}{ \kappa[\widehat{\eta}](\Theta(\tensor{A})) } - 1 \le \kappa[\widehat{\theta}|_{\Var{N}}](\tensor{A}).
\end{align}

The proof will be completed by bounding $\kappa[\widehat{\eta}](\Theta(\tensor{A}))$ from above. As Riemannian metric on $\Var{N} \times \widehat{S}$ we choose the product metric of the natural Riemannian metric on $\Var{N}$, which is inherited from the ambient $\R^{n_1\times n_2\times2}\simeq\R^{n_1 n_2 2}$, and the Riemannian metric that is the pushforward of the standard Euclidean inner product that $S$ inherits from $\R^{n_1 \times r}$ via the map $\pi : S \to \widehat{S}$, which is also a local isometry by the same arguments as in the proof of \cref{lem_kappa_identity}.
Let $A = [\vect{a}_i]_i \in S$ be a factor matrix of $\tensor{B}=\sum_{i=1}^r \vect{a}_i \otimes \vect{b}_i \otimes \vect{z}_i$, which thus imposes an order on the $\vect{a}_i$'s. Let us denote the other two factor matrices by $B=[\vect{b}_i]_i \in S_{r;n_2}$ (the $\vect{b}_i$'s are in GLP) and $Z = [\vect{z}_i]_i\in \R^{2 \times r}$.
Since $\Var{N}\times\widehat{S}$ is locally isometric to $\Var{N}\times S$, there is a local section $\pi_{A}^{-1}$ of $\pi$. As $\widehat{\Var{M}}$ is locally isometric to $\Var{M}$ via $\widehat{\pi}$, there is also a local section $\widehat{\pi}_\star^{-1}$ that is consistent with $A$ in the sense that
\[
 (\widehat{\pi}_{\star}^{-1} \circ \widehat{\eta})(\tensor{B},\{\vect{a}_1,\ldots,\vect{a}_{r}\}) = \eta\bigl(\tensor{B}, \pi_A^{-1}(\{\vect{a}_1,\ldots,\vect{a}_{r}\}) \bigr),
\]
where $\eta(\tensor{B},A) := A \odot (A^\dagger \tensor{B}_{(1)})^T$. We have that $\kappa[\eta](\tensor{B}, A) = \kappa[\widehat{\eta}](\tensor{B},\{\vect{a}_1,\ldots,\vect{a}_r\})$ because of the local isometries. Hence, we can study $\kappa[\eta](\tensor{B}, A)$ instead.

The derivative of $\eta$ is computed as follows. We note that
\begin{align*}
 (\deriv{A^\dagger}{})(\dot{A})
 = (\deriv{(A^T A)^{-1} A^T}{}) (\dot{A})
 &= (A^T A)^{-1} \dot{A}^T + (A^T A)^{-1} ( \dot{A}^T A + A^T \dot{A} ) (A^T A)^{-1} A^T \\
 &= (A^T A)^{-1} \bigl( \dot{A}^T + ( \dot{A}^T A + A^T \dot{A} ) A^\dagger \bigr),
\end{align*}
where $\dot{A}$ is a tangent vector in $ \Tang{A}{S_{r;n_1}}$. We find that
	\begin{multline*}
	(\deriv{\eta}{(\tensor{B}, A)})(\dot{\tensor{B}},\dot{A})
	= A\odot ( A^\dagger \dot{\tensor{B}}_{(1)})^T + \dot{A} \odot ( A^\dagger \tensor{B}_{(1)} )^T \\ + A \odot \bigl( (A^T A)^{-1} \bigl( \dot{A}^T + ( \dot{A}^T A + A^T \dot{A} ) A^\dagger \bigr) \tensor{B}_{(1)} \bigr)^T.
	\end{multline*}
Now, by definition of the Riemannian metrics
\begin{align} \label{eqn_kappa_gamma2}
\kappa[\eta](\tensor{B},A)
 = \max_{ \Vert \dot{\tensor{B}} \Vert_F^2 + \Vert \dot{A}\Vert_F^2 =1} \Vert (\deriv{\eta}{(\tensor{B}, A)})(\dot{\tensor{B}},\dot{A})\Vert_F.
\end{align}
Let $(\dot{\tensor{B}}, \dot{A})$ be a maximizer of \cref{eqn_kappa_gamma2}. Note that $\| \dot{\tensor{B}} \|_F \le 1$ and $\| \dot{A} \|_F \le 1$. Since $A \odot (A^\dagger \tensor{B}_{(1)})^T$ is a submatrix of $A \otimes (A^\dagger \tensor{B}_{(1)})^T$, it follows that $\| A \odot (A^\dagger \tensor{B}_{(1)})^T \|_F \le \| A \|_F \|A^\dagger \tensor{B}_{(1)}\|_F$.
Exploiting this inequality and the triangle inequality a few times, we obtain
\begin{equation*}
\kappa[\eta](\tensor{B},A) \le
 \Vert A\Vert_F \Vert A^\dagger \dot{\tensor{B}}_{(1)}\Vert_F +  \Vert A^\dagger \tensor{B}_{(1)}\Vert_F + \Vert A\Vert_F \Vert (A^T A)^{-1} \bigl( \dot{A}^T + ( \dot{A}^T A + A^T \dot{A} ) A^\dagger \bigr) \tensor{B}_{(1)}\Vert_F.
\end{equation*}
The right-hand side is a Lipschitz continuous function in $(\tensor{B},A) \in \R^{n_1 \times n_2 \times 2} \times \R^{n_1\times r}$, say with Lipschitz constant $\ell > 0$.

By assumption there is a matrix $A'= [\vect{a}_i']_i$ {with orthonormal columns} with $\Vert A-A'\Vert_F <\nu$. Let $\tensor{B}'$ be the tensor with factor matrices $A'$,$B$, $Z$; that is, $\tensor{B}' := \sum_{i=1}^r \vect{a}_i' \otimes \vect{b}_i \otimes \vect{z}_i$. Then, by the triangle inequality and the computation rules for inner products of rank-1 tensors from \cref{eq:innerprod},
\begin{align*}
 \| \tensor{B}' - \tensor{B} \|_F
\le \sum_{i=1}^r \| (\vect{a}_i -\vect{a}_i' ) \otimes \vect{b}_i \otimes \vect{z}_i \|_F
= \sum_{i=1}^r \| \vect{a}_i -\vect{a}_i' \Vert_F \Vert \vect{b}_i \otimes \vect{z}_i \|_F
\leq  r\nu(1+\nu),
\end{align*}
where the last step is because $\|\vect{b}_i \otimes \vect{z}_i\|_F < 1+\nu$ for each $i$. This shows that
$$
\Vert (\tensor{B}, A)-(\tensor{B}', A')\Vert_F
\leq \sqrt{r^2\nu^2 (1+\nu)^2+\nu^2}
= \nu \sqrt{r^2 (1+\nu)^2+1}.
$$
Assume that $\nu \le 1$ and let us write $L := \ell \sqrt{4 r^2 + 1}$.
Then, using the Lipschitz continuity from above, $\Vert A'\Vert_F = \sqrt{r}$ and $(A')^\dagger = (A')^T$ we find
\begin{equation*}
 \kappa[\eta](\tensor{B},A)
 \leq \sqrt{r} \Vert (A')^T \dot{\tensor{B}}_{(1)}\Vert_F  +  \Vert (A')^T \tensor{B}_{(1)}'\Vert_F + \sqrt{r} \Vert \bigl( \dot{A}^T + ( \dot{A}^T A' + (A')^T \dot{A} ) (A')^T \bigr) \tensor{B}_{(1)}'\Vert_F + \nu L.
\end{equation*}
Recall that for matrices $X,Y$ we have the inequality $\|X Y \|_F \le \min\{ \|X\|_2 \|Y\|_F, \|X\|_F \|Y\|_2 \}$. Observe that $(A')^T \tensor{B}_{(1)}' = B\odot Z$ and $\Vert (A')^T\Vert_2 = 1$. Exploiting these we obtain
\begin{align*}
 \kappa[\eta](\tensor{B},A)
&\le \sqrt{r} + \| B\odot Z \|_F + \sqrt{r} \bigl( \|  \dot{A}^T \tensor{B}_{(1)}' \|_F  + \| ( \dot{A}^TA' + (A')^T \dot{A} ) (B \odot Z)^T \|_F \bigr) + \nu L \\
&\le \sqrt{r} + \| B\odot Z \|_F + \sqrt{r} \| \tensor{B}_{(1)}' \|_2  + 2 \sqrt{r} \| \dot{A}^T A' \|_F \| B \odot Z \|_F  + \nu L.
\end{align*}
{Finally, we have $\| \dot{A}^T A' \|_F \leq \| \dot{A}^T\|_F \| A' \|_2 = 1$.} Then, since $\tensor{B}_{(1)}' = A' (B\odot Z)^T$, we also have $\Vert \tensor{B}_{(1)}' \Vert_2 \leq \Vert B\odot Z\Vert_2 \le \Vert B\odot Z\Vert_F \le \sqrt{r}(1 + \nu)$. This shows
\begin{align*}
 \kappa[\eta](\tensor{B},A)
&\le \sqrt{r} + (1 + 3\sqrt{r})\|B\odot Z\|_F + \nu L \le 10 r.
\end{align*}
where in the last step we assumed that $\nu L \le r$.
Plugging this into \cref{eqn_proof_end} finishes the proof. \qed

\subsection{Proof of Lemma \ref{lem_nodecrease}}
Observe that $\widetilde{B} \odot \widetilde{C}$ can naturally be regarded as a matrix in the space $\R^{n_2n_3 \times r}$. Therefore,
\begin{equation*}
\varepsilon := \min_{\pi \in \mathfrak{S}_r} \| A \odot B \odot C  - (\widetilde{A} \odot \widetilde{B} \odot \widetilde{C})P_\pi \|_F \geq
\min_{\pi \in \mathfrak{S}_r} \min_{M\in \R^{n_2 n_3 \times r}} \| A \odot B \odot C  - (\widetilde{A} \odot M) P_\pi \|_F,
\end{equation*}
where $P_\pi$ is the permutation matrix corresponding to $\pi$. Let $\pi \in \mathfrak{S}_r$ be any permutation. Then,
\[
\min_{M\in \R^{n_2 n_3 \times r}} \| A \odot B \odot C  - (\widetilde{A} \odot M) P_{\pi} \|_F
= \min_{M\in \R^{n_2 n_3 \times r}} \| A \odot B \odot C  - (\widetilde{A} P_{\pi}) \odot M \|_F,
\]
where the last step is because of the definition of the Khatri--Rao product, and because every $M \in \R^{n_2 n_3 \times r}$ can be factored as $(M P_{\pi}^{-1}) P_{\pi}$ since $P_{\pi}$ is invertible.
Let $\vect{m}_1,\ldots,\vect{m}_r$ be the columns of $M$. Then, we have that
\begin{equation}\label{SOS_lem4.4}
\| A \odot B \odot C  - (\widetilde{A} P_{\pi} ) \odot M \|_F^2
= \sum_{i=1}^r \Vert \vect{a}_i\otimes (\vect{b}_i\otimes \vect{c}_i) - \widetilde{\vect{a}}_{\pi_i} \otimes \vect{m}_{i} \Vert^2_F
\end{equation}
is a sum of squares, so that we can minimize each ${\vect{m}}_i$ separately.
The first-order necessary optimality conditions are
\[
  (\widetilde{\vect{a}}_{\pi_i} \otimes I_{n_2 n_3})^T ( \vect{a}_i \otimes (\vect{b}_i \otimes \vect{c}_i) - \widetilde{\vect{a}}_{\pi_i} \otimes \vect{m}_{i}) = 0, \quad i=1,\ldots,r.
\]
Solving for $\vect{m}_{i}$ yields the unique solution ${\vect{m}}_{i} = \langle \widetilde{\vect{a}}_{\pi_i},\vect{a}_i\rangle\, \vect{b}_i\otimes \vect{c}_i.$ Plugging this minimizer into the $i$th term in the right-hand side of \cref{SOS_lem4.4}, we find
\begin{align*}
\Vert (\vect{a}_i-\langle \widetilde{\vect{a}}_{\pi_i}, \vect{a}_i\rangle \widetilde{\vect{a}}_{\pi_i})\otimes\vect{b}_i\otimes \vect{c}_i\Vert_F^2
= \Vert \vect{a}_i-\langle\widetilde{\vect{a}}_{\pi_i},\vect{a}_i\rangle \widetilde{\vect{a}}_{\pi_i}\Vert^2 \Vert\vect{b}_i\otimes \vect{c}_i\Vert_F^2
\ge (1-\nu)^2 \Vert \vect{a}_i-\langle\widetilde{\vect{a}}_{\pi_i},\vect{a}_i\rangle \widetilde{\vect{a}}_{\pi_i} \Vert^2,
\end{align*}
where we used the computation rules for inner products from \cref{eq:innerprod} in the first step, and the assumption that $\| \vect{b}_i \otimes \vect{c}_i \|_F \ge 1-\nu$ in the last step.
From this it follows that
\[
\min_{M \in \R^{n_2 n_3 \times r}} \| A \odot B \odot C  - (\widetilde{A} P_{\pi}) \odot M \|_F^2 \ge (1-\nu)^2 \| A - \widetilde{A} P_{\pi} \operatorname{diag}(\langle \widetilde{\vect{a}}_{\pi_1}, \vect{a}_1 \rangle, \ldots, \langle \widetilde{\vect{a}}_{\pi_r}, \vect{a}_r \rangle) \|_F^2.
\]

Let us define
\(
\zeta_{\pi} := \| A - \widetilde{A} P_{\pi} \operatorname{diag}(\langle \widetilde{\vect{a}}_{\pi_1}, \vect{a}_1 \rangle, \ldots, \langle \widetilde{\vect{a}}_{\pi_r}, \vect{a}_r \rangle) \|_F.
\)
We claim that the minimizer of $\min_{\pi \in \mathfrak{S}_r} \zeta_{\pi}$ equals the minimizer $\pi^*$ of $\min_{\pi \in \mathfrak{S}_r} \| A - \widetilde{A} P_{\pi} \|_F$.
To prove this, we show that $\zeta_{\pi^*} = \min_{\pi \in\mathfrak{S}_r} \zeta_\pi$ by exhibiting an upper bound for $\zeta_{\pi^*}$ that is smaller than a lower bound for $\zeta_{\pi}$ with $\pi \ne \pi^*$.
Note that
\(
\langle \widetilde{\vect{a}}_{\pi_i^*}, \vect{a}_i \rangle = \langle \vect{a}_i - \vect{f}_i, \vect{a}_i \rangle = 1 - \langle \vect{f}_i, \vect{a}_i \rangle.
\)
Hence,
\begin{align*}
\zeta_{\pi^*}
&= \| A - \widetilde{A} P_{\pi^*} + \widetilde{A} P_{\pi^*} \operatorname{diag}(\langle \widetilde{\vect{f}}_{1}, \vect{a}_1 \rangle, \ldots, \langle \widetilde{\vect{f}}_i, \vect{a}_r \rangle) \|_F \\
&\le \| A - \widetilde{A} P_{\pi^*} \|_F + \| \widetilde{A} P_{\pi^*} \operatorname{diag}(\langle \widetilde{\vect{f}}_{1}, \vect{a}_1 \rangle, \ldots, \langle \widetilde{\vect{f}}_i, \vect{a}_r \rangle) \|_F \\
&\le \delta + \| \widetilde{A} \|_F \| P_{\pi^*} \|_2 \| \operatorname{diag}(\langle \widetilde{\vect{f}}_{1}, \vect{a}_1 \rangle, \ldots, \langle \widetilde{\vect{f}}_i, \vect{a}_r \rangle) \|_2 = \delta + \sqrt{r} \max_{1\le i\le r} | \langle \vect{f}_i, \vect{a}_i \rangle | \le \delta( 1 + \sqrt{r}),
\end{align*}
where the last step is due to the Cauchy--Schwartz inequality. Next, we lower bound $\zeta_{\pi'}$ with $\pi' \ne \pi^*$. In this case, there is always some $k$ such that $\pi_k' = \pi_j^*$ with $j \ne k$.
Then,
\[
 \| \vect{a}_k - \langle \widetilde{\vect{a}}_{\pi_k'}, \vect{a}_k \rangle \widetilde{\vect{a}}_{\pi_k'} \|^2
 = \| \vect{a}_k - \langle \widetilde{\vect{a}}_{\pi_j^*}, \vect{a}_k \rangle \widetilde{\vect{a}}_{\pi_j^*} \|^2
 = 1 - \langle \widetilde{\vect{a}}_{\pi_j^*}, \vect{a}_k \rangle^2.
\]
Note that for all $i=1,\ldots,r$ we have that
\[
0 \le \| \vect{a}_i' - \widetilde{\vect{a}}_{\pi_i^*} \|
 = \| \vect{a}_i' - (\vect{a}_i + \vect{f}_i) \|
 \le \| \vect{a}_i' - \vect{a}_i \| + \|\vect{f}_i\|
 \le \nu + \delta,
\]
where $\vect{f}_i := \vect{a}_i - \widetilde{\vect{a}}_{\pi_i^*}$ and where we used $\delta_i := \|\vect{f}_i\| = \|\vect{a}_i - \widetilde{\vect{a}}_{\pi_i^*}\|\le \|A - \widetilde{A}P_{\pi^*}\|_F = \delta$ in the last step. Therefore, we have
\begin{align*}
 |\langle \widetilde{\vect{a}}_{\pi_j^*}, \vect{a}_k \rangle|
 =  |\langle \vect{a}_j' + \vect{f}_j , \vect{a}_k' + (\vect{a}_k - \vect{a}_k') \rangle|
 &\le |\langle \vect{a}_j', \vect{a}_k' \rangle| + |\langle \vect{a}_j', \vect{a}_k - \vect{a}_k' \rangle| + |\langle \vect{f}_j , \vect{a}_k' \rangle| + |\langle \vect{f}_j , \vect{a}_k - \vect{a}_k' \rangle| \\
 &\le 0 + \| \vect{a}_k - \vect{a}_k' \| + \| \vect{f}_j \| + \|\vect{f}_j\| \|\vect{a}_k - \vect{a}_k'\|
 \le \nu + \delta + \nu \delta.
\end{align*}
It follows that we have the following lower bound
\[
 \zeta_{\pi'}^2
 = \sum_{i=1}^r \| \vect{a}_i - \langle \widetilde{\vect{a}}_{\pi_i'}, \vect{a}_i \rangle \widetilde{\vect{a}}_{\pi_i'} \|^2
 \ge \| \vect{a}_j - \langle \widetilde{\vect{a}}_{\pi_j'}, \vect{a}_j \rangle \widetilde{\vect{a}}_{\pi_j'} \|^2
 = 1 - \langle \widetilde{\vect{a}}_{\pi_k^*}, \vect{a}_j \rangle^2
 \ge 1 - (\nu + \delta + \nu\delta)^2.
\]
When both $\nu$ and $\delta$ are sufficiently small, we have
\[
\zeta_{\pi^*} \le (1+\sqrt{r})\delta < \sqrt{1 - (\nu + \delta + \nu\delta)^2} \le \zeta_{\pi'}
\]
for all $\pi' \ne \pi^*$. This indeed proves that $\pi^*$ is also the minimizer of $\min_{\pi\in\mathfrak{S}_r} \zeta_\pi$.

Combining the foregoing results, we find
\[
 \varepsilon^2
 \ge (1-\nu)^2 \min_{\pi \in \mathfrak{S}_r} \zeta_{\pi}^2
 = (1-\nu)^2 \zeta_{\pi^*}^2
 = (1-\nu)^2 \sum_{i=1}^r \| \vect{a}_i - \langle \widetilde{\vect{a}}_{\pi_i^*}, \vect{a}_i \rangle \widetilde{\vect{a}}_{\pi_i^*} \|^2.
\]
As before we have
\(
 \| \vect{a}_i - \langle \widetilde{\vect{a}}_{\pi_i^*}, \vect{a}_i \rangle \widetilde{\vect{a}}_{\pi_i^*} \|^2
 = 1 - \langle \widetilde{\vect{a}}_{\pi_i^*}, \vect{a}_i \rangle^2.
\)
By the law of cosines $\langle \widetilde{\vect{a}}_{\pi_i^*},\vect{a}_i\rangle = 1- \tfrac{1}{2}\delta_i^2$, so that $1-\langle \widetilde{\vect{a}}_{\pi_i^*},\vect{a}_i\rangle^2=\delta_i^2(1-\tfrac{1}{4}\delta_i^2)$. Since $\delta_i \le \delta < 1$, we find
$$
\varepsilon^2 =  \min_{\pi\in\mathfrak{S}_r} \| A \odot B \odot C  - (\widetilde{A} \odot \widetilde{B}\odot \widetilde{C})P_{\pi} \|_F^2 \geq (1-\nu)^2 \sum_{i=1}^r \delta_i^2(1-\tfrac{1}{4}\delta_i^2) \ge \frac{3}{4} (1-\nu)^2 \delta^2,
$$
because $\delta^2 = \sum_{i=1}^r \delta_i^2$. This concludes the proof.\qed

\subsection{Proof of Lemma \ref{lem_opendense}}
Recall that $\rho_Q = \operatorname{Id}_{\R^{n_1}} \otimes \operatorname{Id}_{\R^{n_2}} \otimes Q^T$. Both $\Sec{r}{\Var{S}_{n_1,n_2,n_3}}$ and $\Sec{r}{\Var{S}_{n_1,n_2,2}}$ are generically $r$-identifiable by \cref{lem_identifiable} because of the assumption on~$r$.
The image $\Phi_r( \Var{U}_\epsilon/\mathfrak{S}_r)$ is open because $\Phi_r$ is a diffeomorphism onto its image and $\Var{U}_\epsilon/\mathfrak{S}_r \subset \widehat{\Var{M}}_{r}^{n_1,n_2,n_3}$ is an open submanifold by construction. The key step consists of showing that 
$$\Var{N}^* =  \rho_Q^{-1}(\Var{N}_{r;n_1,n_2,2})\cap \Var{N}_{r;n_1,n_2,n_3}$$
is open dense in $\sigma_r(\Var{S}_{n_1,n_2,n_3})$.
By \cref{prop_image}, we already know that $\Var{N}_{r;n_1,n_2,n_3}$ is open dense, so that it suffices to prove that $\rho_Q^{-1}(\Var{N}_{r;n_1,n_2,2})$ is dense in $\sigma_r(\Var{S}_{n_1,n_2,n_3})$. We show this next.

Let $\tensor{A} \in \Sec{r}{\Var{S}_{n_1,n_2,n_3}}$ be arbitrary. We let $\tensor{B} := \rho_Q(\tensor{A})$ and write
\[
 \tensor{A} = \sum_{i=1}^r  \sten{a}{i}{} \otimes \sten{b}{i}{} \otimes \vect{c}_i \text{ and } \tensor{B} = \sum_{i=1}^r \sten{a}{i}{} \otimes \sten{b}{i}{} \otimes \vect{z}_i, \text{ where } \vect{a}_i\in \R^{n_1},\vect{b}_i\in \R^{n_2},\vect{c}_i\in \R^{n_3},\vect{z}_i\in \R^{2}.
\]
Let us decompose $\vect{c}_i = Q\vect{z}_i + Q^\perp \vect{z}_i'$ where $Q^\perp \in \R^{n_3 \times (n_3-2)}$ is a matrix whose columns form an orthonormal basis of the orthogonal complement of the space spanned by the columns of $Q$ and~$\vect{z}_i'\in \R^{n_3-2}$.
Consider a generic sequence such that
\[
 \lim_{k\to\infty} \sten{a}{i}{(k)} = \sten{a}{i}{},\;
 \lim_{k\to\infty} \sten{b}{i}{(k)} = \sten{b}{i}{},\;\text{ and }
 \lim_{k\to\infty} {\vect{z}}_{i}^{(k)} = {\vect{z}}_{i}.
\]
Note that $\tensor{B}_i^{(k)} :=  \sten{a}{i}{(k)} \otimes \sten{b}{i}{(k)} \otimes \sten{z}{i}{(k)}$ lives in $\Var{S}_{n_1,n_2,2}$ by construction. As the sequence is arbitrary and $\Var{M}_{r;n_1,n_2,2}$ is open dense in $\Var{S}_{n_1,n_2,2}^{\times r}$
by \cref{prop_nice}, we can assume that the sequence is restricted so that all $(\tensor{B}_1^{(k)}, \ldots, \tensor{B}_r^{(k)}) \in \Var{M}_{r;n_1,n_2,2}$. Taking the quotient with the symmetric group $\mathfrak{S}_r$, we get by \cref{prop_permutation}:
\(
\{\tensor{B}_1^{(k)}, \ldots, \tensor{B}_r^{(k)} \} \in \widehat{\Var{M}}_{r;n_1,n_2,2}.
\)
Note that $\Phi_r \bigl( \{ \tensor{B}_1^{(k)}, \ldots, \tensor{B}_r^{(k)} \} \bigr) = \sum_{i=1}^r \tensor{B}_i^{(k)} \in \Var{N}_{r;n_1,n_2,2}$ by \cref{prop_image}. Now, let
\[
 \tensor{A}_i^{(k)} := \sten{a}{i}{(k)} \otimes \sten{b}{i}{(k)} \otimes (Q {\vect{z}}_{i}^{(k)} + Q^\perp {\vect{z}}_i').
\]
Then, $\rho_Q(\tensor{A}_i^{(k)}) = \tensor{B}_i^{(k)}$ so that $\tensor{A}_i^{(k)} \in \rho_Q^{-1}(\Var{N}_{r;n_1,n_2,2})$.
Now observe that
\(
 \lim_{k\to\infty} \sum_{i=1}^r \tensor{A}_i^{(k)} = \tensor{A};
\)
in other words, $\tensor{A} \in \overline{\rho^{-1}(\Var{N}_{r;n_1,n_2,2})}$. Since it was arbitrary, this proves the claim. \qed

\subsection{Proof of Lemma \ref{lem_localisometry}}
Recall from \cref{eqn_psi_star} the map \(\Psi^*_{n_1,n_2,n_3}\) and that it is a diffeomorphism. There is a natural isomorphism between $\R\setminus\{0\} \times \mathbb{S}^+(\R^{n_3})$ and $\R^{n_3}\setminus\{0\}$, so that
\[
\Psi^{**}: \mathbb{S}^+(\R^{n_1}) \times \mathbb{S}^+(\R^{n_2}) \times \R^{n_3}\setminus\{0\} \to \Var{S},\;
(\vect{x}, \vect{y}, \vect{z}) \mapsto \vect{x} \otimes \vect{y} \otimes \vect{z}
\]
also is a diffeomorphism. The reason for introducing $\Psi^{**}$ is that it is difficult to ensure that the tensor $\tensor{O}$ lies in the image of $\Psi^*$. Nevertheless, $\tensor{O}$ lies in the image of $\Psi^{**}$. Since $\Psi^{**}$ is a diffeomorphism, there is a Lipschitz constant $\ell$ so that for all $i=1,\ldots,r$ we have
\[
 \| (\vect{a}_i, \vect{b}_i, \vect{c}_i ) - (\vect{a}_i', \vect{b}_i', \vect{c}_i') \| \le \ell \| \tensor{A}_i - \tensor{O}_i \|_F \le \ell \epsilon,
\]
where the norm on the left-hand side is the standard product norm of the Euclidean norms on $\mathbb{S}(\R^{n_1})$, $\mathbb{S}(\R^{n_2})$, and $\R^{n_3}$.
In particular, this implies:
$$
\Vert A-A'\Vert_F < \sqrt{r}\ell \epsilon,\quad \Vert B-B'\Vert_F < \sqrt{r}\ell \epsilon,\quad \Vert C-C'\Vert_F < \sqrt{r}\ell \epsilon.
$$
Hence, for $S \geq \sqrt{r}\ell$ the first part of the lemma holds.
For the second part, we write $\Delta \vect{b}_i : = \vect{b}_i - \vect{b}_i'$ and $\Delta \vect{c}_i : = \vect{c}_i - \vect{c}_i'$. Then, we have
$$
\vect{b}_i\otimes \vect{c}_i = \vect{b}_i' \otimes \vect{c}_i' + \vect{b}_i' \otimes \Delta \vect{c}_i + \Delta \vect{b}_i \otimes \vect{c}_i' +  \Delta \vect{b}_i \otimes \Delta \vect{c}_i .
$$
By the definition of the odeco tensor $\tensor{O}$ in \cref{eqn_cot}, we have $\Vert \vect{b}_i'\Vert = \Vert \vect{c}_i'\Vert = 1$. Using the triangle inequality and the computation rules for inner products from \cref{eq:innerprod}, we get
\begin{align*}
\bigl| \Vert \vect{b}_i\otimes \vect{c}_i \Vert_F - \Vert \vect{b}_i' \otimes \vect{c}_i' \Vert_F \bigr|
&\le \Vert \vect{b}_i' \otimes \Delta \vect{c}_i \Vert_F + \Vert \Delta \vect{b}_i \otimes \vect{c}_i' \Vert_F +  \Vert \Delta \vect{b}_i \otimes \Delta \vect{c}_i \Vert_F \\
&= \Vert \vect{b}_i' \Vert \Vert \Delta\vect{c}_i \Vert_F + \Vert \Delta\vect{b}_i\Vert \Vert\vect{c}_i' \Vert_F +  \Vert \Delta \vect{b}_i \Vert \Vert \Delta \vect{c}_i \Vert_F \le 2 \ell \epsilon + \ell^2 \epsilon^2.
\end{align*}
Since $\| \vect{b}_i' \otimes \vect{c}_i' \| =1$, taking $S\geq\max\{ (\ell+2) \epsilon \ell, \sqrt{r}\ell\}$ finishes the proof.
\qed


\providecommand{\bysame}{\leavevmode\hbox to3em{\hrulefill}\thinspace}
\providecommand{\MR}{\relax\ifhmode\unskip\space\fi MR }
\providecommand{\MRhref}[2]{%
  \href{http://www.ams.org/mathscinet-getitem?mr=#1}{#2}
}
\providecommand{\href}[2]{#2}

\end{document}